\documentclass{article}

\usepackage[T1,T2A]{fontenc}
\usepackage[utf8]{inputenc}
\usepackage[russian,spanish,english]{babel}

\usepackage[letterpaper,top=2cm,bottom=2cm,left=2.5cm,right=2.5cm,marginparwidth=1.75cm]{geometry}

\usepackage{amsmath}
\usepackage{amsthm}
\usepackage{amsfonts}
\usepackage{graphicx}
\usepackage[colorlinks=true, allcolors=blue]{hyperref}
\usepackage{authblk}
\newtheorem{mythm}{Theorem}

\newtheorem{mylemma}[mythm]{Lemma}
\newtheorem{myremark}[mythm]{Remark}

\DeclareMathOperator{\mydiv}{div}
\DeclareMathOperator{\myspan}{span}
\numberwithin{equation}{section}

\title{Quantum Algorithms for  Multiscale Partial Differential Equations}

\author[1]{Junpeng Hu\thanks{hjp3268@sjtu.edu.cn}}
\author[1,2]{Shi Jin\thanks{shijin-m@sjtu.edu.cn}}
\author[1]{Lei Zhang\thanks{lzhang2012@sjtu.edu.cn}}
\affil[1]{School of Mathematical Sciences, Institute of Natural Sciences, MOE-LSC, Shanghai Jiao Tong University, Shanghai, 200240, P. R. China}
\affil[2]{Shanghai Artificial Intelligence Laboratory, Shanghai, China}

\begin{document}
\maketitle

\begin{abstract}
Partial differential equation (PDE) models with multiple temporal/spatial scales are prevalent in several disciplines such as physics, engineering, and many others. These models are of great practical importance but notoriously difficult to solve due to prohibitively small mesh and time step sizes limited by the scaling parameter and CFL condition. Another challenge in scientific computing could come from curse-of-dimensionality.  In this paper, we aim to provide a quantum algorithm, based on either direct approximations of the original PDEs or their homogenized models,  for prototypical  multiscale problems in partial differential equations (PDEs), including elliptic, parabolic and hyperbolic PDEs. To achieve this, we will lift these problems to higher dimensions and leverage the recently developed Schr\"{o}dingerization based  quantum simulation algorithms \cite{jin2022quantum} to efficiently reduce the computational cost of the resulting high-dimensional and multiscale problems. We will examine the error contributions arising from discretization, homogenization, and relaxation, analyze and compare the complexities of these algorithms   in order to identify the best algorithms in terms of complexities for different equations in different regimes.

\end{abstract}

\tableofcontents

\section{Introduction}
\label{sec:introduction}

Scientific computing faces two major challenges when dealing with high-dimensional problems such as the $N$-body Schr\"{o}dinger equation and multiscale problems such as physical processes in highly heterogeneous media. These problems are difficult to solve using classical numerical methods due to the `curse of dimensionality' or  the need to use prohibitively small mesh sizes and time steps to numerically resolve  the small physical scales. Quantum computing, on the other hand,  has shown great potential in solving partial differential equations (PDEs) in very high space dimension \cite{berry2014high, cao2013quantum, berry2017quantum, childs2021high, childs2020quantum, costa2019quantum, liu2021efficient, jin2022time, jin2022quantum, jin2022timenonlinear}.  In this paper, we focus on  quantum algorithms for  multiscale PDEs. To achieve this, we lift these problems to higher dimensions, via the Sch\"odingerization technique \cite{jin2022quantum}, and take advantage of the reduced computational cost offered by quantum algorithms for high-dimensional problems.


Quantum computing has garnered considerable interest in the past few years owing to its potential to offer up to exponential acceleration compared to classical computational methods. A quantum device is best-suited for Hamiltonian simulation, referred to as {\it quantum simulation}, as it can utilize its own evolution to directly prepare the outputs of Schr\"odinger's equations \cite{feynman2018simulating}. The objective of Hamiltonian simulation is to implement the unitary operation $e^{-iHt}$ with a precision of $\delta>0$, where $t$ is the evolution time and $H$ is a given $m$-qubit Hermitian matrix with a size of $2^m \times 2^m$. The precision $\delta$ is measured by the diamond norm distance.

Lloyd's seminal work \cite{lloyd1996universal} introduced the first explicit quantum simulation algorithm, enabling the simulation of Hamiltonians that consist of local interaction terms. Aharonov and Ta-Shma \cite{aharonov2003adiabatic} proposed an efficient simulation algorithm for the more general class of sparse Hamiltonians, and subsequent works have further improved simulations \cite{childs2004quantum,berry2007efficient,wiebe2011simulating,childs2010relationship,berry2009black,childs2012hamiltonian,berry2015simulating,berry2014exponential,berry2015hamiltonian,berry2020time,zhao2022hamiltonian}. In \cite{berry2015hamiltonian}, Berry et al. combined the quantum walk approach \cite{childs2010relationship,berry2009black} with the fractional-query approach \cite{berry2014exponential}, resulting in a reduced query complexity of $O\left( \tau \frac{\log(\tau/\delta)}{\log\log(\tau/\delta)} \right)$, where $\tau := s \left| H \right|_{\max}t$. This approach had a near-linear dependence on the sparsity $s$ and evolution time $t$, and offered exponential speedup over $\delta$. They also established a lower bound demonstrating that this result is near-optimal for scaling in either $\tau$ or $\delta$ independently.

Achieving quantum speedup for ordinary or partial differential equations (ODEs/PDEs) that are not Schr\"odinger type is of significant importance in scientific and engineering applications. One common approach is to discretize the spatial and temporal domains, transforming the linear PDEs into a system of linear algebraic equations, and then use quantum algorithms to solve the corresponding system \cite{harrow2009quantum,cao2012quantum,childs2017quantum}. This can potentially lead to polynomial or super-polynomial speedups in solving PDEs with quantum algorithms \cite{berry2014high,berry2017quantum,childs2021high}. Jin et al. studied the time complexity of quantum difference methods for linear and nonlinear high-dimensional and multiscale PDEs--in the framework of Asymptotic-Preserving schemes-- in \cite{jin2022time,jin2022quantumobservables,jin2022timenonlinear}. A recent approach called `Schr\"odingerisation' has been proposed by Jin et al. in \cite{jin2022quantum,jin2022quantumdetail} to solve linear PDEs by converting them into a  system of Schr\"odinger equations using a simple transformation called the warped phase transformation. This allows quantum simulations for general linear ODEs and PDEs. An alternative approach is given in \cite{an2023linear}.


We begin with multiscale elliptic equations, which  arise in various applications, such as modeling conduction problems in composite materials with periodic structures. One of the most commonly studied examples is the second-order elliptic equation of divergence form:
\begin{equation}
\begin{aligned}
-\nabla \cdot(A (x) \nabla u (x)) &=f(x), \quad & & x \in D, \
u (x) &=0, & & x \in \partial D,
\end{aligned}
\label{eqn:darcy}
\end{equation}
where $0<\alpha\leq A(x) \leq \beta$ for all $x \in D$, and the forcing term $f\in H^{-1}(D;\mathbb{R})$. The setup $A(x) \in L^\infty$ allows for rough coefficients with fast oscillations, high contrast ratios with $\beta/\alpha\gg 1$, and even a continuum of non-separable scales.

In this paper, we consider the canonical multiscale model with $A\left(x, \frac{x}{\varepsilon_1}, \cdots, \frac{x}{\varepsilon_n}\right)$, which exhibits scale separation in the problem. In this paper by "canonical" we refer to the original models. The setup details can be found in Section \ref{sec:elliptic}, where $n\geq 1$ denotes the number of fine scales. Asymptotic homogenization \cite{bensoussan2011asymptotic} is a powerful analytical technique for solving multiscale PDEs with scale separation. In the classical two-scale homogenization setup with $n=1$, we assume that the coefficient $A(x)$ has the form $A(x) = A(x, x/\varepsilon)$ with $\varepsilon\ll 1$, and $A(\cdot, \cdot)$ is periodic in the second variable. Let $u_\varepsilon$ be the corresponding solution, as $\varepsilon\to 0$, we expect $u_\varepsilon$ to converge to $u_0$, where $u_0$ is the solution of an elliptic equation with homogenized coefficient $A_0$. (Higher order) correctors can be used to obtain higher order approximations of $u_\varepsilon$ in terms of $\varepsilon$. For $n>1$, reiterated homogenization can be performed over a hierarchy of scales to construct the reiterated homogenization model. For a comprehensive understanding of qualitative and quantitative homogenization theory and recent developments, see \cite{allaire1992homogenization,armstrong2016lipschitz,avellaneda1987compactness,bensoussan2011asymptotic,niu2018convergence,shen2017boundary,shen2018periodic,suslina2013homogenization}. We also note that, in the past two decades, significant progress has been made in the field of numerical homogenization \cite{Eming2005analysis,weinan2007heterogeneous,abdulle2012heterogeneous,Hou1999,MalPet:2014,Owhadi2007,eh09}, with the development of methods such as heterogeneous multiscale method (HMM) and multiscale finite element (MsFEM).


Analytical approaches, such as periodic unfolding \cite{cioranescu2008periodic}, suggest that a low-dimensional multiscale PDE can be transformed into a high-dimensional PDE. As a consequence, numerical methods based on tensor representations have emerged as powerful tools for efficiently computing multiscale solutions. Among these methods, the sparse tensor finite element method \cite{Harbrecht2011} and the quantic tensor train (QTT) method \cite{Kazeev2022} are prominent examples. By constructing high-dimensional basis functions that are adapted to the multiscale structures, these methods provide a low-rank tensor representation of the solution, resulting in a significant reduction in computational cost. Such approach can be extended to multiscale parabolic and wave equations as well \cite{tan2019high,donato2012periodic,xia2014high}. 

In the present work, we employ reiterated homogenization \cite{AllaireBriane1996} to obtain the high dimensional homogenization model from the canonical multiscale model, and develop a new quantum algorithm to solve them. The algorithm involves three steps: first, we apply the finite element method to the corresponding model and obtain a linear algebraic system $\mathbf{A}\mathbf{u}=\mathbf{F}$; second, we consider $\mathbf{u}$ as the steady state solution $\mathbf{u}_{\infty}$ of a linear differential equation $\mathbf{u}_t = -\mathbf{A}\mathbf{u} + \mathbf{F}$ and approximate it with the solution $\mathbf{u}(t)$ for large enough $t$; finally, we calculate $\mathbf{u}(t)$ using the quantum simulation of ODEs via Schr\"odingerisation \cite{jin2022quantum}.

By comparing the classical cost and quantum cost of the $(n+1)$-scale canonical elliptic model with dimension $d$ and the corresponding homogenization model, one of our main result is given in the following theorem. Throughout the paper, we will denote $C$ the generic constant and the notation $\tilde{O}$ means some logarithmic terms are ignored.
\begin{mythm}[multiscale elliptic equations]\label{thm:quantum:advantages:elliptic}
Under the same assumptions as in Theorem \ref{thm:reiterated:error} and for multiple scales $\varepsilon_k=\varepsilon_1^k$, $k=1,\cdots,n$, if a precision threshold of $\delta = O(\varepsilon_1)$ is prescribed, then
\begin{enumerate}
    \item quantum advantages are possible on $\varepsilon_1$ when $(n+1)d>3$ for the homogenization elliptic model and $(n+1)(d-1)>2$ for the canonical elliptic model;
    \item either the classical or quantum cost of the homogenization model is smaller than the corresponding classical or quantum cost of the canonical model w.r.t. $\varepsilon_1$;
    \item the quantum algorithm for the homogenization model is most recommended when $(n+1)d>3$.
\end{enumerate}
\end{mythm}

We also study multiscale  parabolic and wave equations. For these problems, however, the theoretical results on the convergence of homogenization solutions for general $(n+1)$-scale cases are still under development.  We will
focus on the 2-scale case $A_\varepsilon(x)=A\left(\frac{x}{\varepsilon}\right)$. The comparison results we obtained between classical and quantum algorithms are presented below.
\begin{mythm}[multiscale parabolic and wave equations]\label{thm:quantum:advantages:parabolicwave}
Under the same assumptions as in Theorem \ref{thm:reiterated:error} and for two-scale case $\varepsilon = \varepsilon_1$, if a precision threshold of $\delta = O(\varepsilon_1)$ is prescribed, then
\begin{enumerate}
    \item quantum advantages are possible on $\varepsilon_1$ when $d \geq 1$ for the parabolic equation and the homogenized wave equation but 
    $d>1$ for the canonical wave equation.
    \item Both the classical or quantum costs for the homogenized and canonical parabolic equations are of the same order over $\varepsilon_1$, while the cost of the homogenized parabolic equation grows more quickly w.r.t. the dimension $d$ compared with its canonical counterpart.
    \item the quantum costs for the homogenized and canonical wave equations are of the same order over $\varepsilon_1$, while the classical cost of the homogenized wave equation has a higher order than  that of the canonical wave equation.
\end{enumerate}
\end{mythm}

 We remark here that for both classical and quantum algorithms, the homogenized models do not provide any advantages over their canonical counterparts, for the case of $n=1$. Here the cost of homogenized model is due to the lift to high dimensions, also it is well-known that for the two scale homogenization, the cost to compute the cell problems and the homogenized model is similar to that of canonical problem, with the reduction coming possibly  from multiple right hand sides \cite{Engquist2008}.

\paragraph{Organization} The paper is structured as follows. In Section \ref{sec:quantum}, we introduce a new quantum subroutine and demonstrate its numerical convergence. Section \ref{sec:elliptic} provides an overview of the canonical multiscale elliptic model and the reiterated homogenization theory, as well as a review of standard finite element methods and an estimation of matrix condition numbers. In Sections \ref{sec:parabolic} and \ref{sec:hyperbolic}, we extend the analysis to multiscale parabolic PDEs and wave equations, respectively. We will analyze the complexities for all three equations, as presented in Section \ref{sec:complexity}, and conclude in Section \ref{sec:conclusion} with a discussion of our results.

\section{Quantum Algorithms}
\label{sec:quantum}

It is widely accepted that solving high-dimensional problems using classical computers can be challenging. However, with the help of quantum algorithms, lifting a low-dimensional partial differential equation to higher dimensions may prove beneficial.   In this context, our aim is to demonstrate the advantages of quantum algorithm for the solution of multiscale problems using the Schr\"{o}dingerisation method, which was proposed not only for solving linear PDEs/ODEs but also  some linear algebra problems by iterative methods \cite{JinLiu-LA}.

We begin by obtaining the linear algebraic equations using finite element methods. This process is discussed in the end of this section and the details of finite element formulations are presented in Section \ref{sec:elliptic}, Section \ref{sec:parabolic} and Section \ref{sec:hyperbolic} for elliptic, parabolic and wave equations respectively. 
We can directly solve the multiscale equations using finite element methods. Alternatively, for problems with scale separation, we can use the approach proposed in \cite{AllaireBriane1996} to construct a high-dimensional homogenization model and use tensor product finite element spaces to solve the resulting high-dimensional problem. Once this has been done, one is left with the task of solving the linear system $\mathbf{A}\mathbf{u}=\mathbf{F}$.

We introduce a quantum algorithm to solve this problem. The solution $\mathbf{u}$ can be interpreted as the steady-state solution $\mathbf{u}_{\infty}$ of the following equation:
\begin{equation}\label{eqn:quantum:model}
    \frac{d\mathbf{u}(t)}{dt} = -\mathbf{A}\mathbf{u}(t) + \mathbf{F},
\end{equation}
which can be solved by quantum simulation via Schr\"odingerisation\cite{jin2022quantum} provided that $\mathbf{A}$ is {\it positive
semi-definite} for stability. To see this, we rewrite \eqref{eqn:quantum:model} as 
\begin{equation}
    \left\{\begin{aligned}
        &\frac{d\mathbf{\tilde{u}}(t)}{dt} = -\mathbf{\tilde{A}}\mathbf{\tilde{u}}(t), \\
        &\mathbf{\tilde{u}}(0) = \mathbf{\tilde{u}}_0,
    \end{aligned}\right., \quad \mathbf{\tilde{u}} = \begin{bmatrix}
    \mathbf{\tilde{u}} \\
    v
    \end{bmatrix}, \quad \mathbf{\tilde{A}} = \begin{bmatrix}
    \mathbf{A} & \mathbf{F}(t) \\
    \mathbf{0}^T & 0
    \end{bmatrix}, \quad \mathbf{\tilde{u}}_0 = \begin{bmatrix}
    \mathbf{\tilde{u}}_0 \\
    1
    \end{bmatrix}.
\end{equation}
where $\mathbf{\tilde{A}}$ can be decomposed into a Hermitian term and an anti-Hermitian term
\begin{equation}
    \mathbf{\tilde{A}} = \mathbf{H}_1 + i\mathbf{H}_2, \quad \mathbf{H}_1 = \frac{\mathbf{\tilde{A}} + \mathbf{\tilde{A}}^\dagger}{2}, \quad \mathbf{H}_2 = \frac{\mathbf{\tilde{A}} - \mathbf{\tilde{A}}^\dagger}{2i}.
\end{equation}
Using the warped phase transformation $\mathbf{v}(t,p)=e^{-p}\mathbf{\tilde{u}}(t)$ for $p>0$ and symmetrically extending to $p<0$ as in \cite{jin2022quantumdetail}, the ODEs are then transferred to a Hamiltonian system after the discrete Fourier transformation
\begin{equation}\label{eqn:quantum:schrodingermodel}
    i\frac{d}{dt} \mathbf{w}(t) = (\mathbf{H}_1\otimes \mathbf{D} + \mathbf{H}_2\otimes \mathbf{I})\mathbf{w} =: \mathbf{H}_{total} \mathbf{w},
\end{equation}
where $\mathbf{D}$ is the diagonal matrix representation of the momentum operator $-i\partial_x$ in the original variables and $\mathbf{w}$ collects all the grid values of $\mathbf{v}(t,p)$ with grid size $\Delta p$, defined by
\begin{equation}
    \mathbf{w} = [\mathbf{w}_{0};\mathbf{w}_{1};\cdots;\mathbf{w}_{K}], \quad \mathbf{w}_{k} = \mathbf{v}(t,p_k), \quad p_k = -p_{\max} + k\Delta p, \quad k=0,\cdots,K=2p_{\max}/\Delta p.
\end{equation}
It is easy to check that 
\begin{equation}
    s(\mathbf{H}_{total}) = O(s(\mathbf{\tilde{A}})) = O(s(\mathbf{A})), \quad \|\mathbf{H}_{total}\|_{\max} \leq \|\mathbf{H}_1\|_{\max}/\Delta p + \|\mathbf{H}_2\|_{\max} = O(\|\mathbf{\tilde{A}}\|_{\max}/\Delta p).
\end{equation}
In the situations we have been considering in the previous sections, $\|\mathbf{\tilde{A}}\|_{\max} = \|\mathbf{A}\|_{\max}$ since the maximum value of $\mathbf{F}$ is smaller than $\|\mathbf{A}\|_{\max}$.

The following Lemma elucidates the computational complexity of solving  \eqref{eqn:quantum:schrodingermodel}.

\begin{mylemma}{\cite[Lemma 3.1]{jin2022quantumdetail}}\label{thm:quantum:hamiltonian}
An $s$-sparse Hamiltonian $H_{total}$ acting on $m_{H_{total}}$ qubits can be simulated within error $\delta$ with
\begin{equation}
    \mathcal{O}\left( \tau \frac{\log(\tau/\delta)}{\log\log(\tau/\delta)} \right)
\end{equation}
queries and 
\begin{equation}
    \mathcal{O}\left( \tau(m_{H_{total}}+\log^{2.5}(\tau/\delta)) \frac{\log(\tau/\delta)}{\log\log(\tau/\delta)} \right) = \mathcal{O}(\tau m_{H_{total}} \cdot \mathrm{polylog})
\end{equation}
additional $2$-qubits gates, where $\tau=s\|H_{total}\|_{\max}t$ and $t$ is the evolution time, and
\begin{equation}
    {\mathrm{polylog}} \equiv \log^{2.5}(\tau/\delta) \frac{\log(\tau/\delta)}{\log\log(\tau/\delta)}.
\end{equation}
This result is near-optimal.
\end{mylemma}

Finally, we need to estimate the error between $\mathbf{u}(t)$ and the steady state $\mathbf{u}_{\infty}$, which is demonstrated in the following theorem.

\begin{mythm}\label{thm:quantum:exp:error}
Given $\mathbf{A}$ in \eqref{eqn:quantum:model} with eigenvalues $\lambda_{\max}(\mathbf{A}) \geq \cdots \lambda_{\min}(\mathbf{A}) > 0$, then $\mathbf{u}(t)$ converges to the steady state $\mathbf{u}_{\infty}$ in the sense
\begin{equation}
    \|\mathbf{u}(t) - \mathbf{u}_{\infty}\|_{2} \leq e^{-\lambda_{\min}t} \|\mathbf{u}_0 - \mathbf{u}_{\infty}\|_{2},
\end{equation}
where $\|\cdot\|_2$ denotes the vector $l_2$ norm normalized by its dimension, i.e., $\|\mathbf{u}\|_2^2 = \sum_{i=1}^{N}\mathbf{u}_i^2/N$ for $\mathbf{u}\in\mathbb{R}^N$. Also, $\mathbf{w}(t)$ converges to the steady state $\mathbf{w}_{\infty}$ in the sense
\begin{equation}
    \|\mathbf{w}(t) - \mathbf{w}_{\infty}\|_{2} \leq e^{-\lambda_{\min}t} \|\mathbf{w}(0) - \mathbf{w}_{\infty}\|_{2}.
\end{equation}
\end{mythm}

\begin{proof}
The relaxation error decays as
\begin{equation}
\begin{aligned}
    \frac{d}{dt}\|\mathbf{u}(t) - \mathbf{u}_{\infty}\|_{2}^2 &= \frac{2}{N} (\mathbf{u}(t) - \mathbf{u}_{\infty})^T (-\mathbf{A}\mathbf{u}(t) + F) \\
    &= \frac{2}{N} (\mathbf{u}(t) - \mathbf{u}_{\infty})^T (-\mathbf{A}\mathbf{u}(t) + \mathbf{A}\mathbf{u}_{\infty}) \\
    &\leq -2\lambda_{\min}\|\mathbf{u}(t) - \mathbf{u}_{\infty}\|_2^2.
\end{aligned}
\end{equation}
Applying Gr\"onwall's inequality, we obtain
\begin{equation}
\|\mathbf{u}(t) - \mathbf{u}_{\infty}\|_2^2 \leq e^{-2\lambda_{\min}t}\|\mathbf{u}(0) - \mathbf{u}_{\infty}\|_2^2.
\end{equation}
The convergence of $\mathbf{w}(t)$ is equivalent to $\mathbf{u}(t)$ since
\begin{equation}\label{eqn:quantum:errorw}
\begin{aligned}
    \|\mathbf{w}(t) - \mathbf{w}_{\infty}\|_{2}^{2} &= \frac{1}{K+1} \sum_{k=0}^{K} \left\|\mathbf{v}(t,p_{k}) - \mathbf{v}_{\infty}(p_{k})\right\|_{2}^2 \\
    &= \frac{1}{K+1}\sum_{k=0}^{K} e^{-2|p_k|} \left\| \mathbf{\tilde{u}}(t) - \mathbf{\tilde{u}}_{\infty} \right\|_{2}^{2} \\
    &=C \|\mathbf{u}(0) - \mathbf{u}_{\infty}\|_2^2,
\end{aligned}
\end{equation}
where $C = \frac{N}{(N+1)(K+1)}\sum_{k=0}^{K} e^{-2|p_k|} = O(1)$.

\end{proof}

Now let's return to the application of finite element methods to solve multiscale problems. For the elliptic equation discussed in Section \ref{sec:elliptic}, we naturally obtain the form $\mathbf{A}\mathbf{u}=\mathbf{F}$. For the parabolic equation in Section \ref{sec:parabolic}, we obtain an ODE $\mathbf{M}d\mathbf{u}/dt = -\mathbf{A}\mathbf{u}+\mathbf{F}$. Due to the presence of the matrix $\mathbf{M}$, we can not apply quantum simulation via Schr\"odingerisation immediately. After time discretization and collection of approximate solutions at all times, this is transformed into a higher-dimensional ($N_T$ times larger) system $\mathcal{A}\mathbf{U}=\mathcal{F}$. For the wave equation in Section \ref{sec:hyperbolic}, we obtain $\mathbf{M}d^2\mathbf{u}/dt^2 = -\mathbf{A}\mathbf{u}+\mathbf{F}$ and introduce a new variable $[\mathbf{u};\mathbf{v}]$, where $\mathbf{v}=d\mathbf{u}/dt$, then treat it similarly as the parabolic case to obtain $\mathcal{A}\mathbf{U}=\mathcal{F}$.   We note that in this case, the final output is the quantum state of $\mathbf{u}$ that includes the linear combinations of $\mathbf{u} $ at {\it all} time steps, while for the elliptic
case the final output is that quantum state of $\mathbf{u}$ only at the output time. This will result in different costs (much less for the elliptic case) of quantum measurements if one wants to extra classical data at the output time.

In the upcoming sections, we will present  detailed formulations for multiscale elliptic, parabolic, and wave equations, respectively. We will analyze the key parameters $\|\mathbf{A}\|_{\max}$ and $\lambda_{\min}(\mathbf{A})$ (or $\|\mathcal{A}\|_{\max}$ and $\lambda_{\min}(\mathcal{A})$) that affect the performance of the quantum algorithm. Furthermore, we will also examine the error contributions arising from discretization, homogenization, and relaxation. A key point to emphasize is that the matrix $\mathbf{A}$ or $\mathcal{A}$ must be positive semi-definite, which will be demonstrated by the eigenvalue estimates in the following sections. 

\begin{myremark}
For the canonical parabolic equation \eqref{eqn:parabolic:model} and wave equation \eqref{eqn:hyperbolic:model}, one can perform finite difference methods instead of finite element methods. For instance, the central difference methods provide an ODE in the form of $d\mathbf{u}(t)/dt = -\mathbf{A}\mathbf{u}+\mathbf{F}$. This ODE can be directly used for quantum simulation via Schr\"odingerisation after checking that $\mathbf{A}$ is positive definite. Nevertheless, we can not obtain an ODE in such form for the homogenized equation \eqref{eqn:two-scale:parabolic:homogenizedmodel} and \eqref{eqn:two-scale:hyperbolic:homogenizedmodel}.
\end{myremark}

\section{Multiscale Elliptic  PDEs}
\label{sec:elliptic}

In this section, we consider the canonical elliptic model with coefficient 
\begin{equation}
  A_\varepsilon(x):=A\left( x, \frac{x}{\varepsilon_1}, \cdots, \frac{x}{\varepsilon_n} \right) \quad \forall x \in \Omega,
\end{equation}
where the set of microstructure length scales is denoted by $\varepsilon=\{\varepsilon_k, k=1,\dots,n\}$, and the reference unit cell is $Y = [0 ,1]^d$. We suppose that the spatial domain $\Omega \subseteq \mathbb{R}^d$, with a sequence of periodicity over $Y_k :=\varepsilon_k Y$. The elliptic tensor $A_\varepsilon(x) \in L^\infty(\Omega)^{d^2}$, 
where $A(x,y_1,\dots,y_k)$ is $Y_k$-periodic with respect to the variable $y_k$, and satisfies the following condition, for some $\beta\geq\alpha>0$
\begin{equation}\label{eqn:assump:elliptic}
    \alpha |\xi|^2 \leq A(x,y_1,\dots,y_n)\xi\cdot\xi \leq \beta |\xi|^2 \quad \text{a.e. in } \Omega \times Y_1\times \cdots \times Y_k \text{ for any } \xi \in \mathbb{R}^d.
\end{equation}
The goal is to find $u_\varepsilon$, the solution of 
\begin{equation}\label{eqn:canonical:model}
  \left\{ \begin{aligned}
  & - \mydiv A_\varepsilon \nabla u_\varepsilon = f \quad &\text{in } \Omega, \\
  & u_\varepsilon = 0 \quad & \text{on } \partial \Omega,
  \end{aligned}\right.
\end{equation}
where $f\in L^2(\Omega)$ is a given source term. Under the above mentioned standard hypothesis on the tensor $A_\varepsilon$, this problem is known to have a unique solution in $H_0^1(\Omega)$. 

The computational cost to fully resolve the solution to \eqref{eqn:canonical:model} at the finest scale $\epsilon_n$ with classical numerical methods such as finite element methods, is typically prohibitively high which scales with $(1/\varepsilon_n)^d$. To overcome this limitation, we introduce the homogenization approach, which yields approximate governing equations that are solvable independently of the small parameter $\varepsilon$. Specifically, homogenization of \eqref{eqn:canonical:model} involves an asymptotic analysis of \eqref{eqn:canonical:model} as $\varepsilon$ tends to zero.

\subsection{Two-scale homogenization model}
To better demonstrate the homogenization approach, we begin with a two-scale model with coefficient $A(x, x/\varepsilon)$ and use the formal two-scale asymptotic expansion to formulate the homogenization model. Given the assumptions on the coefficients, we seek a solution $u_{\varepsilon}$ in the form
\begin{equation}\label{eqn:two-scale:expansion}
    u_{\epsilon}(x) = u_0(x,x/\varepsilon) + \varepsilon u_1(x,x/\varepsilon) + \varepsilon^2 u_2(x,x/\varepsilon) + \cdots,
\end{equation}
where the functions $u_j(x,y)$ are defined on $\Omega\times \mathbb{R}^d$ and $Y$-periodic in $y$ for any fixed $x\in\Omega$. Note that if $\phi_\varepsilon(x)=\phi(x,y)$ with $y=x/\varepsilon$, then $\partial_{x_j}\phi_\varepsilon = \partial_{x_j}\phi + \partial_{y_j}\phi / \varepsilon$.

By substituting \eqref{eqn:two-scale:expansion} into \eqref{eqn:canonical:model} and identifying the powers of $\varepsilon$, it follows that
\begin{equation}\label{eqn:two-scale:power}
    \begin{aligned}
        &O(\varepsilon^{-2}): &-\mydiv_{y} A(x,y)\left(\nabla_y u_0(x,y)\right) = 0, \\
    &O(\varepsilon^{-1}): &-\mydiv_{y} A\left( \nabla_x u_0(x,y) + \nabla_y u_1(x,y) \right) -\mydiv_{x} A(x,y)\left(\nabla_y u_0(x,y)\right) = 0, \\
    &O(\varepsilon^{0}): &-\mydiv_{y} A\left( \nabla_x u_1(x,y) + \nabla_y u_2(x,y) \right) -\mydiv_{x} A(x,y)\left(\nabla_x u_0(x,y)+\nabla_y u_1(x,y)\right) = f(x).
    \end{aligned}
\end{equation}
Exploiting the periodicity of $u_0(x,y)$ in $y$ and the ellipticity of $A(x,y)$, one can deduce that $\nabla_{y} u_0(x,y) \equiv 0$, which implies that $u_0(x,y)$ does not depend on $y$, i.e.,
\begin{equation}
    u_0(x,y) \equiv u_0(x).
\end{equation}

By integrating the $O(\varepsilon_1^{0})$ terms over $y\in Y$, one can obtain the two-scale homogenized system by collecting the $O(\varepsilon^{-1})$ and $O(\varepsilon^{0})$ terms in \eqref{eqn:two-scale:power}.
\begin{equation}\label{eqn:two-scale:homogenizedmodel}
    \left\{ \begin{aligned}
    &-\mydiv_{y} A(x,y)\left( \nabla_x u_0(x) + \nabla_y u_1(x,y) \right) = 0, \\
    &-\mydiv_{x} \left[ \int_{Y}  A(x,y)\left( \nabla_x u_0(x) + \nabla_y u_1(x,y) \right) dy \right] = f(x), \\
    &u_1 \text{ is $Y$-periodic with respect to $y$}.
    \end{aligned}\right.
\end{equation}


If $f\in L^2(\Omega)$, it is known that as $\varepsilon\rightarrow 0$, $u_\varepsilon$ converges to $u_0$ weakly in $H^1(\Omega)$ and strongly in $L^2(\Omega)$.  Here we recall some rigorous error estimates of two-scale expansions for the Dirichlet problems from \cite{bensoussan2011asymptotic}.

\begin{mythm}{\cite{bensoussan2011asymptotic,Hoang2005high}}\label{thm:two-scale:error}
Given $u_\varepsilon$ the solution of \eqref{eqn:canonical:model} and $u_0,u_1$ the solution of \eqref{eqn:two-scale:homogenizedmodel}, we can assume that there exists $p>2$ such that $\partial_{x_i} u_0 \in L^p(\Omega)$, then
\begin{equation}
    \left\| u_\varepsilon(x) - u_0(x) \right\|_{L^\infty(\Omega)} \leq C\varepsilon,
\end{equation}
\begin{equation}
    \left\| u_\varepsilon(x) - u_0(x) - \varepsilon \eta_\varepsilon u_1\left(x, \frac{x}{\varepsilon}\right) \right\|_{H_0^1(\Omega)} \leq C \varepsilon^{1/2p^\prime}, \quad \frac{1}{p} + \frac{1}{p^\prime} = 1,
\end{equation}
where $\eta_\varepsilon$ is an $\varepsilon$-smoothing operator. In particular, assuming that $A(x,y)\in C^{\infty} (\bar{\Omega}, C_{\#}^{\infty}(Y))_{sym}^{d\times d}$ and that the homogenized solution $u(x)$ belongs to $H^2(\Omega)$, then
\begin{equation}
    \left\| u_\varepsilon(x) - u_0(x) - \varepsilon u_1\left(x, \frac{x}{\varepsilon}\right) \right\|_{H^1(\Omega)} \leq C \varepsilon^{1/2}.
\end{equation}
\end{mythm}

\subsection{Reiterated homogenization}
For the canonical multiscale model with $n\geq 2$ fine scales, Allaire and Briane proposed the so-called reiterated homogenization method, developed in \cite{AllaireBriane1996} based on the pioneering work \cite{allaire1992homogenization,Nguetseng1989}. Assume that the elliptic tensor $A_{\varepsilon}$ has multiple scales such that 
\begin{equation}
    A_{\varepsilon}(x) = A\left( x, \frac{x}{\varepsilon_1},\dots,\frac{x}{\varepsilon_n} \right),
\end{equation}
where $A(x,y_1,\dots,y_n)$ is $Y$-periodic with respect to each variable $y_k$. Each of these scales is microscopic in the sense that
\begin{equation}
    \lim_{\varepsilon\rightarrow 0} \varepsilon_k = 0, \quad \lim_{\varepsilon\rightarrow 0} \frac{\varepsilon_{k+1}}{\varepsilon_{k}} = 0, \quad 1\leq k\leq n-1.
\end{equation}
The so-called $(n+1)$-scale homogenized system (composed of $(n+1)$ p.d.e.) is written as
\begin{equation}\label{eqn:reiterated:homogenizedmodel}
    \left\{ \begin{aligned}
    &-\mydiv_{y_n} A\left( \nabla_x u_0(x) + \sum_{j=1}^n \nabla_{y_j} u_j \right) = 0, \\
    &-\mydiv_{y_{k}} \left[ \int_{Y_{k+1}}\dots\int_{Y_{n}}  A\left( \nabla_x u_0(x) + \sum_{j=1}^n \nabla_{y_j} u_j \right) dy_{k+1}\dots dy_{n} \right] = 0, \quad 1\leq k \leq n-1,\\
    &-\mydiv_{x} \left[ \int_{Y_{1}}\dots\int_{Y_{n}}  A\left( \nabla_x u_0(x) + \sum_{j=1}^n \nabla_{y_j} u_j \right) dy_{1}\dots dy_{n} \right] = f,\\
    &u_j \text{ is $Y$-periodic with respect to $y_k$}.
    \end{aligned}\right.
\end{equation}
Allaire and Briane proved the multiscale convergence of $u_\varepsilon$ with the following theorem.
\begin{mythm}{\cite{AllaireBriane1996}}
Assume that the solution $(u_0,u_1,\cdots,u_n)$ of the $(n+1)$-scale homogenized problem \eqref{eqn:reiterated:homogenizedmodel} is smooth, say $u_k\in L^2[\Omega;C_{\#}^1(Y_1\times\cdots\times Y_k)]$ for all $k\in\{1,\cdots,n\}$. Then, one has the following corrector result
\begin{equation}
    \left[ u_\varepsilon(x) - u_0(x) - \sum_{k=1}^{n}\varepsilon_ku_k\left( x,\frac{x}{\varepsilon_1},\cdots,\frac{x}{\varepsilon_k} \right) \right] \rightarrow 0 \quad \text{strongly in } H^1(\Omega).
\end{equation}

\end{mythm}


For a quantitative estimate in reiterated homogenization, we refer the reader to \cite{niu2020quantitative} for a recent result.

\begin{mythm}{\cite{niu2020quantitative}}\label{thm:reiterated:error}
Let $\Omega$ be a bounded $C^{1,1}$ domain in $\mathbb{R}^d$. Assume that A is elliptic, periodic and Lipschitz continuous. For $F \in L^2(\Omega; \mathbb{R}^m) $ and $f \in H^{3/2}(\partial \Omega; \mathbb{R}^m)$, let $u_{\varepsilon} \in H^1(\Omega; \mathbb{R}^m)$ be the solution of \eqref{eqn:canonical:model} and $u_0$ the solution of the homogenized problem \eqref{eqn:reiterated:homogenizedmodel}. Then
\begin{equation}
    \|u_{\varepsilon} - u_0\|_{L^2(\Omega)} \leq C \{ \varepsilon_1 + \varepsilon_2/\varepsilon_1 + \cdots + \varepsilon_n/\varepsilon_{n-1}\} \| u_0 \|_{H^2(\Omega)},
\end{equation}
where $C$ depends at most on $d$, $m$, $n$, $\mu$, $L$, and $\Omega$.
\end{mythm}

To the best of the authors' knowledge, obtaining corrector results for multiscale homogenization is generally challenging.

\subsection{Classical numerical methods}
\label{sec:classical}

This section introduces the finite element method as a prototypical classical numerical method for both the canonical multiscale model and the homogenization models. To solve the derived algebraic equation, we will employ the conjugate gradient (CG) method. According to \cite{harrow2009quantum}, this method requires $O(\sqrt{\kappa}\log(1/\delta))$ matrix-vector multiplications, each taking $O(Ns)$ time, resulting in a total running time of $O(Ns\sqrt{\kappa}\log(1/\delta))$ for a positive definite matrix. Here, $s$ denotes the sparsity number, which represents the maximum number of nonzero entries per row, $\kappa$ is the condition number of the matrix, and $\delta$ is the expected error bound of the algorithm. We also provide the discretization error estimates and the properties of the finite element matrix, which will be utilized to calculate the complexities of the methods in Section \ref{sec:complexity}.
 
\subsubsection{Canonical multiscale model}\label{sec:numerical:canonical}
For simplicity, we choose $\Omega=[0,1]^d$. To solve \eqref{eqn:canonical:model} numerically using the finite element method, we seek an approximate solution $u_\varepsilon^h \in V_h$. Here, $V_h$ is the tensor product space of piecewise linear functions on a regular grid with grid size $h=1/(N+1)$, which can be viewed as a special case of the tensor product B-spline space \cite{hollig2003finite}. The basis of $V_h$ is given by $\phi_{\mathbf{i}}(x) = \prod_{l=1}^d \varphi_{i_l}(x^l)$, where $\mathbf{i}=(i_1,\cdots,i_d)$, and 
\begin{equation}
    V_h = \myspan\left\{ \phi_{\mathbf{i}}(x), x\in [0,1]^d \right\} = \underbrace{V_{h,0} \otimes \cdots \otimes V_{h,0}}_{d}, \quad V_{h,0} = \myspan \left\{ \varphi_{i}(x), x\in [0,1] \right\}.
\end{equation}
We then define the mass matrix $\mathbf{M}$, stiffness matrix $\mathbf{K}$, operator matrix $\mathbf{A}$, and force vector $\mathbf{F}$ as follows:
\begin{equation}\label{eqn:matrix:canonical:MK}
\begin{aligned}
    M_{\mathbf{i},\mathbf{i}^\prime} &= \int_\Omega \phi_{\mathbf{i}}(x)\phi_{\mathbf{i}^\prime}(x) dx, \quad \mathbf{M} = \left[M_{\mathbf{i},\mathbf{i}^\prime}\right]_{N^d\times N^d}, \\
    K_{\mathbf{i},\mathbf{i}^\prime} &= \int_\Omega \nabla\phi_{\mathbf{i}}(x)\cdot\nabla\phi_{\mathbf{i}^\prime}(x) dx, \quad \mathbf{K} = \left[K_{\mathbf{i},\mathbf{i}^\prime}\right]_{N^d\times N^d}, \\
    A_{\mathbf{i},\mathbf{i}^\prime} &= \int_\Omega A\left(x,\frac{x}{\varepsilon}\right) \nabla\phi_{\mathbf{i}}(x)\cdot\nabla\phi_{\mathbf{i}^\prime}(x) dx, \quad \mathbf{A} = \left[A_{\mathbf{i},\mathbf{i}^\prime}\right]_{N^d\times N^d}, \\
    F_{\mathbf{i}^\prime} &= \int_\Omega f(x)\phi_{\mathbf{i}^\prime}(x) dx, \quad \mathbf{F} = \left[F_{\mathbf{i}^\prime}\right]_{N^d\times 1}.
\end{aligned}
\end{equation}
Then the approximate solution can be obtained by solving
\begin{equation}\label{eqn:FEM:canonical:model}
    \mathbf{A} \mathbf{u_\varepsilon} = \mathbf{F}, \quad u_\varepsilon^h := \sum_{i_1,\cdots,i_d=1}^{N} u_{\varepsilon,\mathbf{i}} \phi_{\mathbf{i}}, \quad \mathbf{u_{\varepsilon}} = \left[u_{\varepsilon,\mathbf{i}}\right]_{N^d\times 1}.
\end{equation}

The following theorem presents the error estimate for the finite element method.
\begin{mythm}\label{thm:canonical:error}
Let $u_\varepsilon\in H^2(\Omega)$ be the solution of \eqref{eqn:canonical:model} and $u_\varepsilon^h \in V_h$ be the solution of \eqref{eqn:FEM:canonical:model}. Then we have
\begin{equation}
    \left\| u_\varepsilon - u_\varepsilon^h \right\|_{H^1(\Omega)} \leq Ch(\|u_\varepsilon\|_{H^2(\Omega)} + \|f\|_{H^0(\Omega)}),
\end{equation}
\begin{equation}
    \left\| u_\varepsilon - u_\varepsilon^h \right\|_{L^2(\Omega)} \leq Ch^2(\|u_\varepsilon\|_{H^2(\Omega)} + \|f\|_{H^0(\Omega)}).
\end{equation}
\end{mythm}
Note that the estimate blows up as $1/\varepsilon$ as $\varepsilon$ approaches 0, as $|u_\varepsilon |_{H^2(\Omega)}=O(1/\varepsilon)$ \cite{hou2003numerical}. Therefore, to ensure the convergence of the finite element method, we require a prohibitively small mesh size $h<\varepsilon$. To estimate the numerical complexity, we need the following properties of matrix $\mathbf{A}$.
\begin{mythm}\label{thm:canonical:conditionnumber}
Given matrix $\mathbf{A}$ defined above, with $A(x,y)$ under the assumption \eqref{eqn:assump:elliptic}, and considering the piecewise linear element $\varphi_{i_k}$ along dimension $k$, the sparsity $s(\mathbf{A})$ and condition number $\kappa(\mathbf{A})$ satisfy
\begin{equation}
    s(\mathbf{A}) = 3^d, \quad \|A\|_{\max} = O(dh^{d-2}), \quad \kappa(\mathbf{A}) = O(3^ddh^{-2}).
\end{equation}
\end{mythm}

\begin{proof}
The sparsity and maximum value of $\mathbf{A}$ is obvious. To estimate the condition number $\kappa(\mathbf{A})$, we start from the ellipticity assumption \eqref{eqn:assump:elliptic}, then for any  $\mathbf{u_{\varepsilon}}\in\mathbb{R}^{N^d}$,
\begin{equation}
    \mathbf{u_{\varepsilon}}^T\mathbf{A}\mathbf{u_{\varepsilon}} = \int_{\Omega} A\left(x,\frac{x}{\varepsilon}\right)\nabla u_{\varepsilon}^h(x)\cdot\nabla u_{\varepsilon}^h(x) dx \leq \beta \int_{\Omega} \nabla u_{\varepsilon}^h(x)\cdot\nabla u_{\varepsilon}^h(x) dx = \beta \mathbf{u_{\varepsilon}}^T\mathbf{K}\mathbf{u_{\varepsilon}},
\end{equation}
and similarly $\mathbf{u_{\varepsilon}}^T\mathbf{A}\mathbf{u_{\varepsilon}} \geq \alpha \mathbf{u_{\varepsilon}}^T\mathbf{K}\mathbf{u_{\varepsilon}}$. It follows by the fact that $\kappa(\mathbf{A}) \leq \frac{\beta}{\alpha} \kappa(\mathbf{K})$.

The condition numbers of the mass matrix $\mathbf{M}$ and stiffness matrix $\mathbf{K}$ are well known, see \cite{hollig2003finite,fried1973l2}. For completeness, we include the following derivation for piecewise linear element, i.e., for any $l=1,\cdots,d$, $i_l=1,\cdots,N$, 
\begin{equation}
    \varphi_{i_l}(x^l) = \frac{x^l-(i_l-1)h}{h} \mathbb{I}_{i_l-1} + \frac{(i_l+1)h-x^l}{h} \mathbb{I}_{i_l}, \quad \mathbb{I}_{i_l} = [i_lh, (i_l+1)h].
\end{equation}
To begin with, let $d=1$, then
\begin{equation}
\begin{aligned}
    M_{i,i} &= 2\int_{0}^{h} \frac{x^2}{h^2} dx = \frac{2}{3}h, \quad M_{i,i+1} = M_{i,i-1} = \int_{0}^{h} \left(1-\frac{x}{h}\right) \frac{x}{h} dx = \frac{h}{6}, \quad M_{i,j} = 0, \text{ otherwise}, \\
    K_{i,i} &= 2\int_{0}^{h} \frac{1}{h^2} dx = \frac{2}{h}, \quad K_{i,i+1} = K_{i,i-1} = \int_{0}^{h} -\frac{1}{h^2} dx = -\frac{1}{h}, \quad K_{i,j} = 0, \text{ otherwise}. \\
\end{aligned}
\end{equation}
By the Gershgorin circle theorem \cite{gershgorin1931uber}, we obtain 
\begin{equation}
    \lambda_{\max}(\mathbf{M_1}) \leq h, \quad \lambda_{\max}(\mathbf{K_1}) \leq \frac{4}{h}, \quad \lambda_{\min}(\mathbf{K_1})\geq \pi^2\lambda_{\min}(\mathbf{M_1}) \geq \pi^2\frac{h}{3},
\end{equation}
where the subscript index $1$ denotes $d=1$ and we used the Poincaré inequality with constant $\pi^2$ \cite{bebendorf2003note},
\begin{equation}\label{eqn:M1K1:conditionnumber}
    \int_{\Omega} |\nabla v^h(x)|^2 dx \geq \pi^2 \int_{\Omega} |v^h(x)|^2 dx \quad \Rightarrow \quad  \mathbf{u_{\varepsilon}}^T\mathbf{K}\mathbf{u_{\varepsilon}} \geq \pi^2 \mathbf{u_{\varepsilon}}^T\mathbf{M}\mathbf{u_{\varepsilon}}.
\end{equation}

For the case $d\geq 2$, it is easy to calculate
\begin{equation}\label{eqn:matrix:MK}
    \mathbf{M} = \bigotimes_{l=1}^{d} \mathbf{M_1}, \quad \mathbf{K} = \mathbf{K_1}\otimes\mathbf{M_1}\otimes\cdots\otimes\mathbf{M_1} + \cdots + \mathbf{M_1}\otimes\cdots\otimes\mathbf{M_1}\otimes\mathbf{K_1},
\end{equation}
which leads to the following results
\begin{equation}\label{eqn:MK:conditionnumber}
    \lambda_{\max}(\mathbf{K}) \leq 4dh^{d-2}, \quad\lambda_{\min}(\mathbf{K}) \geq \pi^2 \left(\frac{h}{3}\right)^d, \quad  \kappa(\mathbf{A}) \leq \frac{4\beta}{\alpha \pi^2} 3^d d h^{-2}.
\end{equation}
\end{proof}

\subsubsection{Two-scale homogenization model}
For the two-scale homogenization model, $u_0(x)$ and $u_1(x,y)$ are approximated in the space $V_h$ and $V_h\otimes V_h$ respectively. Similar as in Section \ref{sec:numerical:canonical}, we define the stiffness matrix $\mathbf{\tilde{K}}$, operator matrix $\mathbf{\tilde{A}}$ and force vector $\mathbf{\tilde{F}}$ as follows,
\begin{equation}\label{eqn:matrix:two-scale:MK}
\begin{aligned}
    \tilde{K}_{\mathbf{j}^\prime \mathbf{k}^\prime, \mathbf{j}\mathbf{k}} &= \int_\Omega\int_Y (\phi_{\mathbf{j}}\phi_{\mathbf{j}^\prime})(x) (\nabla_y\phi_{\mathbf{k}} \cdot \nabla_y\phi_{\mathbf{k}^\prime})(y) dy dx, \quad 
    \tilde{K}_{\mathbf{i}^\prime, \mathbf{i}} = \int_\Omega\int_Y (\nabla_x\phi_{\mathbf{i}} \cdot \nabla_x\phi_{\mathbf{i}^\prime})(x) dy dx, \\
    \tilde{K}_{\mathbf{j}^\prime \mathbf{k}^\prime, \mathbf{i}} &= \int_\Omega\int_Y (\phi_{\mathbf{j}^\prime}\nabla_x\phi_{\mathbf{i}})(x) \cdot \nabla_y\phi_{\mathbf{k}^\prime}(y) dy dx = 0, \quad 
    \tilde{K}_{\mathbf{i}^\prime, \mathbf{j}\mathbf{k}} = \int_\Omega\int_Y (\phi_{\mathbf{i}}\nabla_x\phi_{\mathbf{i}^\prime})(x) \cdot \nabla_y\phi_{\mathbf{j}}(y) dy dx = 0, \\
    \tilde{A}_{\mathbf{j}^\prime \mathbf{k}^\prime, \mathbf{j}\mathbf{k}} &= \int_\Omega\int_Y A(x,y)(\phi_{\mathbf{j}}\phi_{\mathbf{j}^\prime})(x) (\nabla_y\phi_{\mathbf{k}} \cdot \nabla_y\phi_{\mathbf{k}^\prime})(y) dy dx, \quad 
    \tilde{A}_{\mathbf{i}^\prime, \mathbf{i}} = \int_\Omega\int_Y A(x,y)(\nabla_x\phi_{\mathbf{i}} \cdot \nabla_x\phi_{\mathbf{i}^\prime})(x) dy dx, \\
    \tilde{A}_{\mathbf{j}^\prime \mathbf{k}^\prime, \mathbf{i}} &= \int_\Omega\int_Y A(x,y)(\phi_{\mathbf{j}^\prime}\nabla_x\phi_{\mathbf{i}})(x) \cdot \nabla_y\phi_{\mathbf{k}^\prime}(y) dy dx, \quad 
    \tilde{A}_{\mathbf{i}^\prime, \mathbf{j}\mathbf{k}} = \int_\Omega\int_Y A(x,y)(\phi_{\mathbf{j}}\nabla_x\phi_{\mathbf{i}^\prime})(x) 
    \cdot \nabla_y\phi_{\mathbf{k}}(y) dy dx, \\
    \tilde{F}_{\mathbf{i}^\prime} &= \int_\Omega f(x)\phi_{\mathbf{i}^\prime}(x) dx, \quad \mathbf{i}=(i_1,\cdots,i_d), \quad \mathbf{j}=(j_1,\cdots,j_d), \quad \mathbf{k}=(k_1,\cdots,k_d),\\
\end{aligned}
\end{equation}
with
\begin{equation}
    u_0^h = \sum_{1\leq \mathbf{i} \leq N} u_{0,\mathbf{i}} \phi_{\mathbf{i}}(x), \quad u_1^h = \sum_{0\leq \mathbf{j} \leq N+1 \atop 1 \leq \mathbf{k} \leq N} u_{1,\mathbf{j}\mathbf{k}} \phi_{\mathbf{j}}(x)\phi_{\mathbf{k}}(y).
\end{equation}

\begin{myremark}
We remark here on the range of indices $\mathbf{i}, \mathbf{j}$ and $\mathbf{k}$. The choice $1\leq \mathbf{i} \leq N$ is clear since $u_0\equiv 0$ on $\partial \Omega$, which can be obtained from the convergence $u_\varepsilon \rightharpoonup u_0$ in $H^1(\Omega)$. However, $u_1|_{\partial \Omega}$ needs to be calculated in terms of $\nabla_x u_0(x)$, therefore we let $0\leq \mathbf{j} \leq N+1$. Besides, due to the periodicity, there is a unique $u_1$ up to a constant w.r.t. $y$ and we choose the constant to force $u_1|_{\partial Y}\equiv 0$, which leads to $1\leq \mathbf{k}\leq N$.
\end{myremark}

With the definition \eqref{eqn:matrix:two-scale:MK}, one gets the following $((N+2)^dN^{d}+N^d)\times((N+2)^dN^{d}+N^d)$ matrices and $((N+2)^dN^{d}+N^d)\times 1$ vectors,
\begin{equation}\label{eqn:matrix:two-scale:uv}
    \mathbf{\tilde{K}} = \begin{bmatrix}
    \left[\tilde{K}_{\mathbf{j}^\prime \mathbf{k}^\prime, \mathbf{j}\mathbf{k}}\right] & \left[\tilde{K}_{\mathbf{j}^\prime \mathbf{k}^\prime, \mathbf{i}}\right] \\ 
    \left[\tilde{K}_{\mathbf{i}^\prime, \mathbf{j}\mathbf{k}}\right] & \left[\tilde{K}_{\mathbf{i}^\prime, \mathbf{i}^\prime}\right] \\ 
    \end{bmatrix}, \quad 
    \mathbf{\tilde{A}} = \begin{bmatrix}
    \left[\tilde{A}_{\mathbf{j}^\prime \mathbf{k}^\prime, \mathbf{j}\mathbf{k}}\right] & \left[\tilde{A}_{\mathbf{j}^\prime \mathbf{k}^\prime, \mathbf{i}}\right] \\ 
    \left[\tilde{A}_{\mathbf{i}^\prime, \mathbf{j}\mathbf{k}}\right] & \left[\tilde{A}_{\mathbf{i}^\prime, \mathbf{i}^\prime}\right] \\ 
    \end{bmatrix}, \quad 
    \mathbf{\tilde{F}} = \begin{bmatrix}
    \mathbf{0} \\ 
    \left[\tilde{F}_{\mathbf{i}^\prime}\right]\\ 
    \end{bmatrix}, \quad 
    \mathbf{\tilde{u}} = \begin{bmatrix}
    \left[u_{1,\mathbf{j}\mathbf{k}}\right]\\ 
    \left[u_{0,\mathbf{i}}\right]\\ 
    \end{bmatrix}.
\end{equation}
The approximate solution can be obtained by solving
\begin{equation}\label{eqn:FEM:two-scale:model}
    \mathbf{\tilde{A}} \mathbf{\tilde{u}} = \mathbf{\tilde{F}}.
\end{equation}

From the ellipticity of $A(x,y)$, we have $\|u_0\|_{H^2(\Omega\times Y)}+\|u_1\|_{H^2(\Omega\times Y)} \simeq \| f\|_{H^0(\Omega)}$. Applying Theorem \ref{thm:canonical:error}, it follows by
\begin{equation}
    \|u_0-u_0^h\|_{L^2(\Omega\times Y)} + \|u_1 - u_1^h\|_{L^2(\Omega\times Y)} \leq Ch^2.
\end{equation}
The properties of matrix $\mathbf{\tilde{A}}$ is given below.
\begin{mythm}\label{thm:two-scale:conditionnumber}
Given matrix $\mathbf{\tilde{A}}$ defined as above, with $A(x,y)$ under the assumption \eqref{eqn:assump:elliptic}, and considering the tensor products of piecewise linear finite elements
\begin{equation}
    \phi_{\mathbf{i}} = \prod_{l=1}^d \varphi_{i_l}(x^l), \quad \phi_{\mathbf{j}} = \prod_{l=1}^d \frac{1}{\sqrt{h}}\varphi_{j_l}(x^l), \quad 
    \phi_{\mathbf{k}} = \prod_{l=1}^d \varphi_{k_l}(y^l),
\end{equation}
then the sparsity $s(\mathbf{\tilde{A}})$ and condition number $\kappa(\mathbf{\tilde{A}})$ satisfy
\begin{equation}
    s(\mathbf{\tilde{A}}) = 3^{2d} + 3^{d}, \quad \|A\|_{\max} = O(dh^{d-2}), \quad \kappa(\mathbf{\tilde{A}}) = O(3^{2d}dh^{-2}).
\end{equation}
\end{mythm}

\begin{proof}
The proof is similar as the proof of Theorem \ref{thm:canonical:conditionnumber}. The sparsity and maximum value of $\mathbf{\tilde{A}}$ is obvious. According to the ellipticity assumption \eqref{eqn:assump:elliptic}, for any  $\mathbf{\tilde{u}}\in\mathbb{R}^{((N+2)^dN^{d}+N^d)}$,
\begin{equation}
    \alpha \mathbf{\tilde{u}}^T\mathbf{\tilde{K}}\mathbf{\tilde{u}} \leq \mathbf{\tilde{u}}^T\mathbf{\tilde{A}}\mathbf{\tilde{u}} \leq \beta \mathbf{\tilde{u}}^T\mathbf{\tilde{K}}\mathbf{\tilde{u}} \quad \Rightarrow \quad \kappa(\mathbf{\tilde{A}}) \leq \frac{\beta}{\alpha} \kappa(\mathbf{\tilde{K}}).
\end{equation}
From the definition \eqref{eqn:matrix:two-scale:MK}, it is clear that $\mathbf{\tilde{K}}$ can be written as
\begin{equation}
    \mathbf{\tilde{K}} = \begin{bmatrix}
    \frac{1}{h^d}\mathbf{M}\otimes\mathbf{K} & \mathbf{0} \\
    \mathbf{0} & \mathbf{K}
    \end{bmatrix}.
\end{equation}
Recalling the condition number of $\mathbf{M}$, $\mathbf{K}$ in \eqref{eqn:M1K1:conditionnumber} and \eqref{eqn:MK:conditionnumber}, we obtain
\begin{equation}
    \lambda_{\max}(\mathbf{\tilde{K}}) \leq \lambda_{\max}(\mathbf{K}), \quad  \lambda_{\min}(\mathbf{\tilde{K}}) \geq \frac{\lambda_{\min}(\mathbf{K})}{3^d}, \quad \kappa(\mathbf{\tilde{A}}) \leq 3^d\kappa(\mathbf{A}) = O(3^{2d}dh^{-2}).
\end{equation}
\end{proof}

\subsubsection{Reiterated homogenization}
The definitions of $\mathbf{\tilde{K}}, \mathbf{\tilde{A}}, \mathbf{\tilde{F}}$ can be directly extended to the case $n\geq 2$ by approximating $u_k(x,y_1,\dots,y_k)$ with $u_k^h(x,y_1,\dots,y_k) \in V^{k+1}_h := \underbrace{V_h\otimes \dots \otimes V_h}_{k+1}$. After a similar derivation, we obtain an error estimate 
\begin{equation}
    \|u_0-u_0^h\|_{L^2(\Omega\times Y_1 \times \dots \times Y_n)} + \sum_{k=1}^n\|u_k - u_k^h\|_{L^2(\Omega\times Y_1\times \dots \times Y_n)} \leq Ch^2,
\end{equation}
and the properties of matrix $\mathbf{\tilde{A}}$.
\begin{mythm}\label{thm:reiterated:conditionnumber}
Let $\mathbf{\tilde{A}}$ be the matrix obtained from applying the finite element method to \eqref{eqn:reiterated:homogenizedmodel}, with $A(x,y)$ satisfying the assumption \eqref{eqn:assump:elliptic}. We consider the tensor product of piecewise linear finite elements for $u_k$, where $k$ ranges from 0 to $n$,
\begin{equation}
    \phi_{\mathbf{i}}^k = \prod_{l=1}^d \varphi_{i_l}^k(y_k^l), \quad \phi_{\mathbf{i}}^{r} = \prod_{l=1}^d \frac{1}{\sqrt{h}}\varphi_{i_l}^r(y_r^l), \quad r = 0,\cdots,k-1,
\end{equation}
where we denote $y_0=x$. The stiffness matrix $\mathbf{\tilde{K}}$ can be written as 
\begin{equation}
    \mathbf{\tilde{K}} = \begin{bmatrix}
        \frac{1}{h^d}\mathbf{M}\otimes\cdots\otimes\frac{1}{h^d}\mathbf{M}\otimes\mathbf{K} & \mathbf{0} & \mathbf{0} & \mathbf{0}\\
        \mathbf{0} & \cdots & \mathbf{0} & \mathbf{0}\\
        \mathbf{0} & \mathbf{0} & \frac{1}{h^d}\mathbf{M}\otimes\mathbf{K} & \mathbf{0} \\
        \mathbf{0} & \mathbf{0} & \mathbf{0} & \mathbf{K}
    \end{bmatrix}.
\end{equation}
The sparsity $s(\mathbf{\tilde{A}})$ and condition number $\kappa(\mathbf{\tilde{A}})$ satisfy
\begin{equation}
    s(\mathbf{\tilde{A}}) = \sum_{k=1}^{n+1}3^{kd}, \quad \|A\|_{\max} = O(dh^{d-2}), \quad \kappa(\mathbf{\tilde{A}}) = O(3^{(n+1)d}dh^{-2}).
\end{equation}
\end{mythm}
\begin{proof}
The proof is similar as the proof of Theorem \ref{thm:two-scale:conditionnumber}.
\end{proof}
\section{Multiscale Parabolic PDEs}
\label{sec:parabolic}



\subsection{Canonical model and homogenization model}

We consider multiscale parabolic equations of the following type
\begin{equation}\label{eqn:parabolic:model}
  \left\{ \begin{aligned}
  & \frac{\partial u_\varepsilon}{\partial t} - \mydiv A_\varepsilon \nabla u_\varepsilon = f \quad & (x,t) \in \Omega_T := \Omega \times [0,T], \\
  & u_\varepsilon(x,t) = 0 \quad & (x,t) \in \partial \Omega \times [0,T], \\
  & u_\varepsilon(x,0) = u^0(x), \quad & x \in \Omega,
  \end{aligned}\right.
\end{equation}
where the coefficient $A_\varepsilon = A(x/\varepsilon)$ is dependent of $t$ and satisfies the same assumptions as in Section \ref{sec:elliptic}. Then 
\begin{equation}\label{eqn:two-scale:parabolic:homogenizedmodel}
    \left\{ \begin{aligned}
    &-\mydiv_{y} A(y)\left( \nabla_x u_0(x,t) + \nabla_y u_1(x,y,t) \right) = 0, \\
    &\frac{\partial u_0(x,t)}{\partial t}  -\mydiv_{x} \left[ \int_{Y}  A(y)\left( \nabla_x u_0(x,t) + \nabla_y u_1(x,y,t) \right) dy \right] = f(x,t), \\
    &u_1 \text{ is $Y$-periodic with respect to $y$},
    \end{aligned}\right.
\end{equation}
defines $u_0$, the limit of $u_\varepsilon$ \cite{bensoussan2011asymptotic} and we have the following results.

\begin{mythm}{\cite{meshkova2016homogenization}}\label{thm:parabolic:reiterated:error}
Suppose that $\Omega\subset\mathbb{R}^d$ is a bounded domain of class $C^{1,1}$ and $A_\varepsilon$ satisfies the assumptions \eqref{eqn:assump:elliptic}. Let $u_\varepsilon$ be the solution of \eqref{eqn:parabolic:model} and $u_0$ be the solution of \eqref{eqn:two-scale:parabolic:homogenizedmodel} for $u^0(x) \in L^2(\Omega)$ and $f\in L^p([0,T];L^2(\Omega))$, $0<T\leq\infty$ with some $1<p\leq\infty$. Then for $0<\varepsilon<\varepsilon_0$ and $0<t<T$, we have
\begin{equation}
    \left\| u_\varepsilon(\cdot,t) - u_0(\cdot,t) \right\|_{L^2(\Omega)} \leq C \varepsilon (t+\varepsilon^2)^{-1/2} \left\|u^0\right\|_{L^2(\Omega)} + C_p \theta(\varepsilon,p) \left\|f\right\|_{L^p([0,T];L^2(\Omega))},
\end{equation}
where the constants $C$, $C_p$ depend only on $\alpha$, $\beta$, $\Omega$ and $C_p$ also depends on $p$. $\theta(\varepsilon,p)$ is given by
\begin{equation}
    \theta(\varepsilon,p) := \left\{ \begin{aligned}
        &\varepsilon^{2-2/p}, \quad & 1<p<2, \\
        &\varepsilon(|\ln\varepsilon|+1)^{1/2}, \quad & p=2, \\
        &\varepsilon, \quad & 2<p\leq\infty.
    \end{aligned} \right.
\end{equation}
\end{mythm}

Assuming additionally $f\in L^\infty([0,T];L^2(\Omega))$, we have $\left\| u_\varepsilon(\cdot,t) - u_0(\cdot,t) \right\|_{L^2(\Omega)} \leq C\varepsilon$. Both the parabolic equation studied in this section and the wave equation studied in the next section exclusively involve the two-scale case, i.e. $n=1$, $\varepsilon=\varepsilon_1$.


\subsection{FEM for canonical parabolic model \eqref{eqn:parabolic:model}}

Applying finite element method to \eqref{eqn:parabolic:model}, with $\mathbf{M}, \mathbf{A}, \mathbf{F}(t)$ defined similarly as \eqref{eqn:matrix:canonical:MK}, we need to solve 
\begin{equation}\label{eqn:FEM:parabolic:canonical:model}
    \mathbf{M} \frac{d \mathbf{u}_\varepsilon(t)}{dt} + \mathbf{A} \mathbf{u}_\varepsilon(t) = \mathbf{F}(t), \quad u_\varepsilon^h(t) := \sum_{i_1,\cdots,i_d=1}^{N} u_{\varepsilon,\mathbf{i}}(t) \phi_{\mathbf{i}}, \quad \mathbf{u_{\varepsilon}}(t) = \left[u_{\varepsilon,\mathbf{i}}(t)\right]_{N^d\times 1}.
\end{equation}
The implicit Euler time discretization of \eqref{eqn:FEM:parabolic:canonical:model} is
\begin{equation}
    (\mathbf{M} + \Delta t \mathbf{A}) \mathbf{u}_{\varepsilon, j+1} = \mathbf{M} \mathbf{u}_{\varepsilon, j} + \Delta t \mathbf{F}_{j+1},
\end{equation}
where the subscript index $j$ denotes time $t_{j}=j\Delta t$, $0\leq j \leq N_T := T/\Delta t$. Let $\mathbf{U}_{\varepsilon} = [\mathbf{u}_{\varepsilon,1};\cdots;\mathbf{u}_{\varepsilon,N_T}]$, $\mathcal{F} = [\mathbf{F}_{1};\cdots;\mathbf{F}_{N_T}]$, we rewrite it as
\begin{equation}\label{eqn:parabolic:FEM:matrix:Atime}
    \mathcal{A} \mathbf{U}_\varepsilon = \Delta t \mathcal{F}, \quad \mathcal{A} = \mathcal{L}(\mathbf{M}+\Delta t \mathbf{A}, \mathbf{M}),
\end{equation}
where $\mathcal{L}(\mathbf{a}, \mathbf{b})$ is defined as
\begin{equation}
    \mathcal{L}(\mathbf{a}, \mathbf{b}) := \begin{bmatrix}
    \mathbf{a} & & & \\
    -\mathbf{b} & \mathbf{a} & & \\
     & \ddots & \ddots & & \\
     & & -\mathbf{b} & \mathbf{a}
    \end{bmatrix}_{N_T \times N_T}.
\end{equation}
Clearly we have (from Theorem \ref{thm:canonical:conditionnumber})
\begin{equation}
    s(\mathcal{A}) = 3^d + 3^d, \quad \|\mathcal{A}\|_{\max} = O(dh^{d-2}\Delta t + h^d), \quad \kappa(\mathcal{A}) \leq \frac{\lambda_{\max}(\mathbf{M}) + \Delta t\lambda_{\max}(\mathbf{A})}{\lambda_{\min}(\mathbf{M}) + \Delta t\lambda_{\min}(\mathbf{A})} \leq \frac{3^d(h+4d\beta\Delta t/h)}{1+\alpha\pi^2\Delta t} h^{-1}.
\end{equation}
Therefore $\|\mathcal{A}\|_{\max} = O(dh^{d-1})$ and $\kappa(\mathcal{A}) = O(3^d d h^{-1})$ if $\Delta t = O(h)$.

\subsection{FEM for homogenization model \eqref{eqn:two-scale:parabolic:homogenizedmodel}}

Applying finite element method to \eqref{eqn:two-scale:parabolic:homogenizedmodel}, with $\mathbf{\tilde{A}}, \mathbf{\tilde{u}}(t)$ defined similarly as \eqref{eqn:matrix:two-scale:MK} and \eqref{eqn:matrix:two-scale:uv}, we need to solve
\begin{equation}\label{eqn:FEM:parabolic:two-scale:model}
\left\{\begin{aligned}
    \mathbf{\tilde{A}}_{11} \mathbf{u}_1(t) + \mathbf{\tilde{A}}_{12} \mathbf{u}_0(t) &= \mathbf{0}, \\
    \mathbf{M} \frac{d \mathbf{u}_{0}(t)}{dt} + \mathbf{\tilde{A}}_{21} \mathbf{u}_{1}(t) + \mathbf{\tilde{A}}_{22} \mathbf{u}_{0}(t) &= \mathbf{F}(t).
\end{aligned}\right.
\end{equation}
We use the block matrix notations of $\mathbf{\tilde{A}}$. The implicit Euler time discretization of \eqref{eqn:FEM:parabolic:two-scale:model} is
\begin{equation}
\left\{\begin{aligned}
    \mathbf{\tilde{A}}_{11} \mathbf{u}_{1,j+1} + \mathbf{\tilde{A}}_{12} \mathbf{u}_{0,j+1} &= \mathbf{0}, \\
    \Delta t \mathbf{\tilde{A}}_{21} \mathbf{u}_{1,j+1} + (\mathbf{M} + \Delta t \mathbf{\tilde{A}}_{22}) \mathbf{u}_{0,j+1} &= \mathbf{M} \mathbf{u}_{0, j} + \Delta t \mathbf{F}_{j+1},
\end{aligned}\right.
\end{equation}
and can be rewritten as
\begin{equation}
    \mathcal{\tilde{A}} \mathbf{\tilde{U}} = \Delta t \mathcal{\tilde{F}}, \quad \mathcal{\tilde{A}} = \mathcal{L} \left( \mathbf{\tilde{M}} + \Delta t \mathbf{\tilde{A}}, \mathbf{\tilde{M}} \right), \quad \mathbf{\tilde{M}} := \begin{bmatrix} 
    \mathbf{0} & \mathbf{0} \\
    \mathbf{0} & \mathbf{M}
    \end{bmatrix},
\end{equation}
with $\mathbf{\tilde{U}} = [\mathbf{u}_{1,1};\mathbf{u}_{0,1};\cdots;\mathbf{u}_{1,N_T};\mathbf{u}_{0,N_T}]$, $\mathcal{\tilde{F}} = [\mathbf{0};\mathbf{F}_{1};\cdots;\mathbf{0};\mathbf{F}_{N_T}]$. Clearly we have (from Theorem \ref{thm:two-scale:conditionnumber})
\begin{equation}
    s(\mathcal{\tilde{A}}) = 3^{2d} + 3^d + 3^d, \quad \|\mathcal{A}\|_{\max} = O(dh^{d-2}\Delta t + h^d), \quad \kappa(\mathcal{\tilde{A}}) \leq \frac{\lambda_{\max}(\mathbf{\tilde{M}}) + \Delta t\lambda_{\max}(\mathbf{\tilde{A}})}{\lambda_{\min}(\mathbf{\tilde{M}}) + \Delta t\lambda_{\min}(\mathbf{\tilde{A}})} \leq \frac{3^{2d}(h^2/\Delta t+4d\beta)}{\alpha\pi^2} h^{-2}.
\end{equation}
Therefore $\|\mathcal{A}\|_{\max} = O(dh^{d-1})$ and $\kappa(\mathcal{\tilde{A}}) = O(3^{2d} d h^{-2})$ if $\Delta t = O(h)$.

\section{Multiscale Wave Equations}
\label{sec:hyperbolic}


\subsection{Canonical model and homogenization model}

We consider the following prototypical multiscale wave PDE
\begin{equation}\label{eqn:hyperbolic:model}
  \left\{ \begin{aligned}
  & \frac{\partial^2 u_\varepsilon}{\partial t^2} - \mydiv A_\varepsilon \nabla u_\varepsilon = f \quad & (x,t) \in \mathbb{R}^d\times [0,T], \\
  & u_\varepsilon(x,0) = u^0(x), \quad \partial_t u_\varepsilon(x,0) \equiv 0 & x \in \mathbb{R}^d, 
  \end{aligned}\right.
\end{equation}
where the coefficient $A_\varepsilon = A(x/\varepsilon)$ is independent of $t$ and satisfies the same assumptions as in Section \ref{sec:elliptic}. For simplicity, we make the assumptions that $u^0$ and $f$ are compactly supported, and that for $t\in[0,T_0]$, $u_\varepsilon$ is also compactly supported in $\Omega\subset\mathbb{R}^d$. The first order homogenization model
\begin{equation}\label{eqn:two-scale:hyperbolic:homogenizedmodel}
    \left\{ \begin{aligned}
    &-\mydiv_{y} A(y)\left( \nabla_x u_0(x,t) + \nabla_y u_1(x,y,t) \right) = 0, \\
    &\frac{\partial^2 u_0(x,t)}{\partial t^2}  -\mydiv_{x} \left[ \int_{Y}  A(y)\left( \nabla_x u_0(x,t) + \nabla_y u_1(x,y,t) \right) dy \right] = f(x,t), \\
    &u_1 \text{ is $Y$-periodic with respect to $y$},
    \end{aligned}\right.
\end{equation}
defines the limit solution $u_0$ of $u_\varepsilon$ \cite{bensoussan2011asymptotic}. Some previous works have studied the convergence of solutions for long-time homogenization of wave equations \cite{dohnal2015dispersive,allaire2016comparison,benoit2019long}. Here we adapt the first order results from Benoit and Gloria \cite[Theorem 2]{benoit2019long}.
\begin{mythm}\label{thm:hyperbolic:reiterated:order}
    Assume that the first order correctors are bounded (see Hypothesis 1 in \cite{benoit2019long} for details), and  $u_\varepsilon$ is also compactly supported in $\Omega\subset\mathbb{R}^d$ up to time $T_0$.
    Then for all $T_0\geq T\geq 0$
    \begin{equation}
        \sup_{0\leq t\leq T} \|u_\varepsilon - u_0\|_{L^2(\mathbb{R}^d)} \leq C(u^0,f)\left( \varepsilon+\varepsilon T 
        \right).
    \end{equation}
\end{mythm}

\subsection{FEM for canonical wave equations \eqref{eqn:hyperbolic:model}}

Applying finite element method to \eqref{eqn:hyperbolic:model}, with $\mathbf{M}, \mathbf{A}, \mathbf{F}(t)$ defined similarly as \eqref{eqn:matrix:canonical:MK}, we need to solve 
\begin{equation}\label{eqn:FEM:hyperbolic:canonical:model}
    \mathbf{M} \frac{d^2 \mathbf{u}_\varepsilon(t)}{dt^2} + \mathbf{A} \mathbf{u}_\varepsilon(t) = \mathbf{F}(t), \quad u_\varepsilon^h(t) := \sum_{i_1,\cdots,i_d=1}^{N} u_{\varepsilon,\mathbf{i}}(t) \phi_{\mathbf{i}}, \quad \mathbf{u_{\varepsilon}}(t) = \left[u_{\varepsilon,\mathbf{i}}(t)\right]_{N^d\times 1}.
\end{equation}
We define $\mathbf{v}_{\varepsilon} = d \mathbf{u}_{\varepsilon} / dt$ and  \eqref{eqn:FEM:hyperbolic:canonical:model} is transformed to the forced Hamiltonian system
\begin{equation}\label{eqn:FEM:hyperbolic:hamiltonian}
    \left\{ \begin{aligned}
        \frac{d \mathbf{u}_\varepsilon(t)}{dt} &= \mathbf{v}_\varepsilon(t), \\
        \frac{d \mathbf{v}_\varepsilon(t)}{dt} &= -\mathbf{M}^{-1}\mathbf{A}\mathbf{u}_\varepsilon(t) + \mathbf{M}^{-1}\mathbf{F}(t). \\
    \end{aligned} \right.
\end{equation}

The implicit midpoint scheme for \eqref{eqn:FEM:hyperbolic:hamiltonian} is given by:
\begin{equation}
    \left\{ \begin{aligned}
        \mathbf{u}_{\varepsilon,j+1} &= \mathbf{u}_{\varepsilon,j} + \Delta t  \frac{\mathbf{v}_{\varepsilon,j}+\mathbf{v}_{\varepsilon,j+1}}{2}, \\
        \mathbf{v}_{\varepsilon,j+1} &= \mathbf{v}_{\varepsilon,j} - \Delta t \mathbf{M}^{-1}\mathbf{A} \frac{\mathbf{u}_{\varepsilon,j}+\mathbf{u}_{\varepsilon,j+1}}{2} + \Delta t \mathbf{M}^{-1}\mathbf{F}_{j+\frac{1}{2}}, \\
    \end{aligned} \right.
\end{equation}
which can be written as 
\begin{equation}
    \left\{ \begin{aligned}
        \left( \mathbf{M} + \frac{\Delta t^2}{4} \mathbf{A} \right) \mathbf{u}_{\varepsilon,j+1} &= \left( \mathbf{M} - \frac{\Delta t^2}{4} \mathbf{A} \right) \mathbf{u}_{\varepsilon,j} + \Delta t \mathbf{M} \mathbf{v}_{\varepsilon,j} + \frac{\Delta t^2}{2} \mathbf{F}_{j+\frac{1}{2}}, \\
        \frac{\Delta t}{2} \mathbf{A} \mathbf{u}_{\varepsilon,j+1} + \mathbf{M} \mathbf{v}_{\varepsilon,j+1} &= -\frac{\Delta t}{2} \mathbf{A} \mathbf{u}_{\varepsilon,j} + \mathbf{M} \mathbf{v}_{\varepsilon,j} + \Delta t \mathbf{F}_{j+\frac{1}{2}}. \\
    \end{aligned} \right.
\end{equation}

Let $\mathbf{U}_{\varepsilon} = \left[\mathbf{u}_{\varepsilon,1};\mathbf{v}_{\varepsilon,1};\cdots;\mathbf{u}_{\varepsilon,N_T},\mathbf{v}_{\varepsilon,N_T}\right]$, $\mathcal{F} = \left [\frac{\Delta t}{2}\mathbf{F}_{\frac{1}{2}};  \mathbf{F}_{\frac{1}{2}};\cdots;\frac{\Delta t}{2}\mathbf{F}_{N_T-\frac{1}{2}};\mathbf{F}_{N_T-\frac{1}{2}} \right]$, we rewrite it as
\begin{equation}
    \mathcal{A} \mathbf{U}_\varepsilon = \Delta t \mathcal{F}, \quad \mathcal{A} = \mathcal{L}\left(
    \begin{bmatrix} \mathbf{M} + \frac{\Delta t^2}{4} \mathbf{A} & \mathbf{0} \\ \frac{\Delta t}{2} \mathbf{A} & \mathbf{M} \end{bmatrix},
    \begin{bmatrix} \mathbf{M} - \frac{\Delta t^2}{4} \mathbf{A} & \Delta t \mathbf{M} \\ -\frac{\Delta t}{2} \mathbf{A} & \mathbf{M} \end{bmatrix}
    \right).
\end{equation}
Then from Theorem \ref{thm:canonical:conditionnumber} and its proof we have $s(\mathcal{A}) = 4\times3^d$ and 
\begin{equation}
\begin{aligned}
    \lambda_{\max}(\mathcal{A}) &\leq \lambda_{\max}(\mathbf{M}) + \frac{\Delta t^2}{4}\lambda_{\max}(\mathbf{A}) \leq h^d + d\beta \Delta t^2 h^{d-2}, \quad  \lambda_{\min}(\mathcal{A}) = \lambda_{\min}(\mathbf{M}) \geq \frac{h^d}{3^d}, \\
    \kappa(\mathcal{A}) & \leq 3^d (1+d\beta \Delta t^2 h^{-2}), \quad \|\mathcal{A}\|_{\max} = O(dh^{d-2}\Delta t).
\end{aligned}
\end{equation}
Therefore $\|\mathcal{A}\|_{\max} = O(dh^{d-1})$ and $\kappa(\mathcal{A}) = O(3^dd)$ if $\Delta t = O(h)$.

\subsection{FEM for homogenization model \eqref{eqn:two-scale:hyperbolic:homogenizedmodel}}

Applying finite element method to \eqref{eqn:two-scale:hyperbolic:homogenizedmodel}, with $\mathbf{\tilde{A}}, \mathbf{\tilde{u}}(t)$ defined similarly as \eqref{eqn:matrix:two-scale:MK} and \eqref{eqn:matrix:two-scale:uv}, we need to solve
\begin{equation}\label{eqn:FEM:hyperbolic:two-scale:model}
\left\{\begin{aligned}
    \mathbf{\tilde{A}}_{11} \mathbf{u}_1(t) + \mathbf{\tilde{A}}_{12} \mathbf{u}_0(t) &= \mathbf{0}, \\
    \mathbf{M} \frac{d^2 \mathbf{u}_{0}(t)}{dt^2} + \mathbf{\tilde{A}}_{21} \mathbf{u}_{1}(t) + \mathbf{\tilde{A}}_{22} \mathbf{u}_{0}(t) &= \mathbf{F}(t).
\end{aligned}\right.
\end{equation}
We define $\mathbf{v}_{0} = d\mathbf{u}_{0}/dt$, then \eqref{eqn:FEM:hyperbolic:two-scale:model} can be tranformed into 
\begin{equation}\label{eqn:FEM:hyperbolic:two-scale:hamiltonian}
\left\{\begin{aligned}
    &\mathbf{\tilde{A}}_{11} \mathbf{u}_{1}(t) + \mathbf{\tilde{A}}_{12} \mathbf{u}_{0}(t) = \mathbf{0}, \\
    &\frac{d \mathbf{u}_0(t)}{dt} = \mathbf{v}_0(t), \\
    &\frac{d \mathbf{v}_0(t)}{dt} = -\mathbf{M}^{-1}\mathbf{\tilde{A}}_{21} \mathbf{u}_{1}(t) - \mathbf{M}^{-1} \mathbf{\tilde{A}}_{22} \mathbf{u}_{0}(t) + \mathbf{M}^{-1}\mathbf{F}(t). \\
\end{aligned}\right.
\end{equation}
The implicit midpoint rule of \eqref{eqn:FEM:hyperbolic:two-scale:hamiltonian} is
\begin{equation}
\left\{\begin{aligned}
    &\mathbf{\tilde{A}}_{11} \mathbf{u}_{1,j+1} + \mathbf{\tilde{A}}_{12} \mathbf{u}_{0,j+1} = \mathbf{0}, \\
    &\mathbf{u}_{0,j+1} = \mathbf{u}_{0,j} + \Delta t  \frac{\mathbf{v}_{0,j}+\mathbf{v}_{0,j+1}}{2}, \\
    &\mathbf{v}_{0,j+1} = \mathbf{v}_{0,j} - \Delta t \mathbf{M}^{-1}\mathbf{\tilde{A}}_{21} \frac{\mathbf{u}_{1,j}+\mathbf{u}_{1,j+1}}{2} - \Delta t \mathbf{M}^{-1}\mathbf{\tilde{A}}_{22} \frac{\mathbf{u}_{0,j}+\mathbf{u}_{0,j+1}}{2} + \Delta t \mathbf{M}^{-1}\mathbf{F}_{j+\frac{1}{2}}.
\end{aligned}\right.
\end{equation}
It is equivalent to 
\begin{equation}
\left\{\begin{aligned}
    &\mathbf{\tilde{A}}_{11} \mathbf{u}_{1,j+1} + \mathbf{\tilde{A}}_{12} \mathbf{u}_{0,j+1} = \mathbf{0}, \\
    &\frac{\Delta t^2}{4} \mathbf{\tilde{A}}_{21}\mathbf{u}_{1,j+1} + \left(\mathbf{M} + \frac{\Delta t^2}{4} \mathbf{\tilde{A}}_{22} \right) \mathbf{u}_{0,j+1} = -\frac{\Delta t^2}{4} \mathbf{\tilde{A}}_{21} \mathbf{u}_{1,j} + \left(\mathbf{M} - \frac{\Delta t^2}{4} \mathbf{\tilde{A}}_{22} \right) \mathbf{u}_{0,j} + \Delta t \mathbf{M} \mathbf{v}_{0,j} + \frac{\Delta t^2}{2} \mathbf{F}_{j+\frac{1}{2}}, \\
    &\frac{\Delta t^2}{4} \mathbf{\tilde{A}}_{21}\mathbf{u}_{1,j+1} + \frac{\Delta t^2}{4} \mathbf{\tilde{A}}_{22} \mathbf{u}_{0,j+1} + \frac{\Delta t}{2}\mathbf{M} \mathbf{v}_{0,j+1} = -\frac{\Delta t^2}{4} \mathbf{\tilde{A}}_{21} \mathbf{u}_{1,j} - \frac{\Delta t^2}{4} \mathbf{\tilde{A}}_{22} \mathbf{u}_{0,j} + \frac{\Delta t}{2}\mathbf{M} \mathbf{v}_{0,j} + \frac{\Delta t^2}{2} \mathbf{F}_{j+\frac{1}{2}}, \\
\end{aligned}\right.
\end{equation}
and can be written as
\begin{equation}
    \mathcal{\tilde{A}} \mathbf{\tilde{U}} = \frac{\Delta t^2}{2} \mathcal{\tilde{F}}, \quad \mathcal{\tilde{A}} = \mathcal{L}\left(
    \begin{bmatrix} \frac{\Delta t^2}{4} \mathbf{\tilde{A}}_{11} & \frac{\Delta t^2}{4} \mathbf{\tilde{A}}_{12} & \mathbf{0} \\ \frac{\Delta t^2}{4} \mathbf{\tilde{A}}_{21} & \mathbf{M} + \frac{\Delta t^2}{4} \mathbf{\tilde{A}}_{22} & \mathbf{0} \\ \frac{\Delta t^2}{4}\mathbf{\tilde{A}}_{21} & \frac{\Delta t^2}{4} \mathbf{\tilde{A}}_{22} & \frac{\Delta t}{2}\mathbf{M} \end{bmatrix},
    \begin{bmatrix} \mathbf{0} & \mathbf{0} & \mathbf{0} \\ -\frac{\Delta t^2}{4} \mathbf{\tilde{A}}_{21} & \mathbf{M} - \frac{\Delta t^2}{4} \mathbf{\tilde{A}}_{22} & \Delta t \mathbf{M} \\ -\frac{\Delta t^2}{4} \mathbf{\tilde{A}}_{21} & -\frac{\Delta t^2}{4} \mathbf{\tilde{A}}_{22} & \frac{\Delta t}{2}\mathbf{M} \end{bmatrix}
    \right).
\end{equation}
with $\mathbf{\tilde{U}} = \left[\mathbf{u}_{1,1};\mathbf{u}_{0,1};\mathbf{v}_{0,1};\cdots;\mathbf{u}_{1,N_T};\mathbf{u}_{0,N_T};\mathbf{v}_{0,N_T}\right]$, $\mathcal{\tilde{F}} = \left[\mathbf{0};\mathbf{F}_{\frac{1}{2}};\mathbf{F}_{\frac{1}{2}};\cdots;\mathbf{0};\mathbf{F}_{N_T-\frac{1}{2}};\mathbf{F}_{N_T-\frac{1}{2}} \right]$. Then from Theorem \ref{thm:two-scale:conditionnumber} and its proof we have $s(\mathcal{A}) = 2\times(3^{2d}+3^d+3^d)$ and 
\begin{equation}
\begin{aligned}
    \lambda_{\max}(\mathcal{\tilde{A}}) &\leq \lambda_{\max}(\mathbf{M}) + \frac{\Delta t^2}{4}\lambda_{\max}(\mathbf{\tilde{A}}) \leq h^d + d\beta \Delta t^2 h^{d-2}, \quad  \lambda_{\min}(\mathcal{A}) \geq \frac{\Delta t^2}{4} \lambda_{\min}(\mathbf{\tilde{A}}) \geq \frac{\Delta t^2h^d}{4\alpha\pi^2 3^{2d}}, \\
    \kappa(\mathcal{A}) & \leq 4\alpha\pi^2 3^{2d} (h^2/\Delta t^2 + d\beta) h^{-2}, \quad \|\mathcal{A}\|_{\max} = O(dh^{d-2}\Delta t^2).
\end{aligned}
\end{equation}
Therefore $\|\mathcal{A}\|_{\max} = O(dh^{d})$ and $\kappa(\mathcal{\tilde{A}}) = O(3^{2d}dh^{-2})$ if $\Delta t = O(h)$.

\section{Complexity Analysis}
\label{sec:complexity}

The aim of this section is to determine the computational cost associated with obtaining a numerical solution $u_\varepsilon^h$, which satisfies the following condition:
\begin{equation}
    \| u_\varepsilon - u_\varepsilon^h \|_{L^2(\Omega)} \leq \delta,
\end{equation}
where $\delta$ is the threshold and $\varepsilon$ is the microstructure lengthscale. To be precise, we take the quantum solution as an example, denoting  
\begin{equation}
    u_\varepsilon^h := \sum_{1\leq \mathbf{i} \leq N} u_{0,\mathbf{i}}^t \phi_{\mathbf{i}}(x), \quad u_0^h(t) := \sum_{1\leq \mathbf{i} \leq N} u_{0,\mathbf{i}}(t) \phi_{\mathbf{i}}(x),
\end{equation}
with $u_{0,\mathbf{i}}(t)$  the solution of \eqref{eqn:quantum:model}, $w_{0,\mathbf{i}}(t)$ the solution of \eqref{eqn:quantum:schrodingermodel}, $u_{0,\mathbf{i}}^t$ the quantum approximation of $u_{0,\mathbf{i}}(t)$ and $w_{0,\mathbf{i}}^t$ the quantum approximation of $w_{0,\mathbf{i}}(t)$. The error can be divided into different parts, \begin{equation}
    \|u_\varepsilon - u_\varepsilon^h\| \leq \|u_\varepsilon - u_0\| + \|u_0 - u_0^h\| + \| u_0^h - u_0^h(t)\| + \| u_0^h(t) - u_\varepsilon^h\|,
\end{equation}
where $\|\cdot\| = \|\cdot \|_{L^2(\Omega)}$. The first and second terms can be estimated using Theorem \ref{thm:two-scale:error} or Theorem \ref{thm:reiterated:error} and Theorem \ref{thm:canonical:error} respectively. During the quantum simulation, we obtain $w_{0,\mathbf{i}}^t$ first and then reconstruct $u_{0,\mathbf{i}}^t$. Due to the equivalence of error as shown in \eqref{eqn:quantum:errorw}, the third term can be estimated with
\begin{equation}
    \left\| u_0^h - u_0^h(t)\right\| = \left\| \sum_{1\leq \mathbf{i} \leq N} (u_{0,\mathbf{i}} - u_{0,\mathbf{i}}(t)) \phi_{\mathbf{i}}(x) \right\| \leq C \left\| \mathbf{u}(t) - \mathbf{u}_{\infty} \right\|_2 \leq C \left\| \mathbf{w}(t) - \mathbf{w}_{\infty} \right\|_2.
\end{equation}
Similarly, the fourth term is estimated with
\begin{equation}
    \| u_0^h(t) - u_\varepsilon^h\| \leq C \| \mathbf{u}(t) - \mathbf{u}^t \|_2 \leq C \| \mathbf{w}(t) - \mathbf{w}^t \|_2,
\end{equation}
which corresponds to the precision of quantum simulation. We will require each part of the error to be $O(\delta)$. For simplicity, we choose $\varepsilon_k=\varepsilon_1^k, k=1,\cdots,n$ and $\delta = O(\varepsilon_1)$. Therefore the homogenization error is obviously $O(\delta)$ from the convergence results Theorem \ref{thm:reiterated:error}, Theorem \ref{thm:parabolic:reiterated:error} and Theorem \ref{thm:hyperbolic:reiterated:order}.


Table \ref{tab:cost:comparison} provides a summary of the detailed results for comparison purposes.
\begin{table}[htbp]
    \centering
    \begin{tabular}{cccc}
        \hline
        \hline
        $(n+1)$-scale PDE & Classical cost $\mathcal{C}$ & Quantum cost $\mathcal{Q}$ & $\mathcal{C}/\mathcal{Q}$ \\
        \hline
        
         & & & \\
        Elliptic equation & & & \\
        canonical & $\tilde{O}\left(3^{\frac{3d}{2}}d^{\frac{1}{2}} \varepsilon_1^{-\frac{(n+1)(d+1)}{2}}\right)$ & $\tilde{O}\left(3^{2d} d^{\frac{11}{2}} \varepsilon_1^{-(n+2)}\right)$ & $
        \tilde{O}\left(3^{-\frac{d}{2}}d^{-5}\varepsilon_1^{-\frac{(n+1)(d-1)}{2}+1} \right)$ \\
        
        homogenization & $\tilde{O}\left(3^{\frac{3(n+1)d}{2}}d^{\frac{1}{2}}\varepsilon_1^{-\frac{(n+1)d+1}{2}}\right)$ & $\tilde{O}\left(3^{2(n+1)d} n^{\frac{9}{2}}d^{\frac{11}{2}} \varepsilon_1^{-2}\right)$ & $\tilde{O}\left(3^{-\frac{(n+1)d}{2}} n^{-\frac{9}{2}}d^{-5} \varepsilon_1^{\frac{3-(n+1)d}{2} }\right)$ \\
        
         & & & \\
        Parabolic equation ($n=1$) & & & \\
        canonical & $\tilde{O}\left(3^{\frac{3d}{2}}d^{\frac{1}{2}} \varepsilon_1^{-d-\frac{3}{2}}\right)$ & $\tilde{O}\left(3^{2d} d^{\frac{11}{2}} \varepsilon_1^{-2}\right)$ & $
        \tilde{O}\left(3^{-\frac{d}{2}}d^{-5}\varepsilon_1^{-d+\frac{1}{2}}\right)$ \\
        
        homogenization & $\tilde{O}\left(3^{3d}d^{\frac{1}{2}}\varepsilon_1^{-d-\frac{3}{2}}\right)$ & $\tilde{O}\left(3^{4d} d^{\frac{11}{2}} \varepsilon_1^{-2}\right)$ & $\tilde{O}\left(3^{-d} d^{-5} \varepsilon_1^{-d+\frac{1}{2}}\right)$ \\

         & & & \\
        Wave equation ($n=1$) & & & \\
        canonical & $\tilde{O}\left(3^{\frac{3d}{2}}d^{\frac{1}{2}} \varepsilon_1^{-d-1}\right)$ & $\tilde{O}\left(3^{2d} d^{\frac{11}{2}} \varepsilon_1^{-2}\right)$ & $
        \tilde{O}(3^{-\frac{d}{2}}d^{-5}\varepsilon_1^{-d+1})$ \\
        
        homogenization & $\tilde{O}\left(3^{3d}d^{\frac{1}{2}}\varepsilon_1^{-d-\frac{3}{2}}\right)$ & $\tilde{O}\left(3^{4d} d^{\frac{11}{2}} \varepsilon_1^{-2}\right)$ & $\tilde{O}(3^{-d} d^{-5} \varepsilon_1^{-d+\frac{1}{2}})$ \\

        \hline
    \end{tabular}
    \caption{The time complexities of classical and quantum methods for the multiscale models are summarized in the table, where $d$ denotes the spatial dimension, and $n$ denotes the level of multiple scales.}
    \label{tab:cost:comparison}
\end{table}

We note that for $n=1$, the homogenization does not help, or is even worse for parabolic and wave equations. For general $n$-scale case for the  elliptic equations, the classical cost is larger by order $n/2$ w.r.t. $\varepsilon_1$. However the condition number of the canonical equation is $1$ or $2$ orders smaller than the condition number of the homogenized equations for parabolic and wave equations. This will reduce the classical canonical cost by $1/2$ or $1$ order, resulting in the extra costs for the homogenized models. 

The intermediate steps of the complexity calculation for each model are presented below. We begin by expressing the parameters, such as mesh size or time step, that need to be set specifically as functions of $\varepsilon$ and $\delta$ (up to a constant). Next, we substitute these parameters into the corresponding complexity formula of classical or quantum algorithms to obtain the final result. It is important to note that these calculations do not include the cost of initial data encoding or quantum measurement costs at the output.

\subsection{Multiscale elliptic model}
\begin{itemize}
    \item \textbf{Canonical model - Classical method} 
    \begin{itemize}
        \item discretisation error: 
        \begin{equation}
            O(h^2/\varepsilon_{n}) \leq \delta \quad \Rightarrow \quad h = O(\sqrt{\delta \varepsilon_{n}}).
        \end{equation}
        
        \item classical cost: 
        \begin{equation}
            \mathcal{C}_{can} = O\left(s(\mathbf{A})N^d\sqrt{\kappa(\mathbf{A})}\log(1/\delta)\right) = \tilde{O}\left(3^{\frac{3d}{2}}d^{\frac{1}{2}} (\delta\varepsilon_{n})^{-\frac{d+1}{2}}\right) = \tilde{O}\left(3^{\frac{3d}{2}} d^{\frac{1}{2}} \varepsilon_1^{-\frac{(n+1)(d+1)}{2}}\right).
        \end{equation}
    \end{itemize}
    
    \item \textbf{Canonical model - Quantum method} 
    \begin{itemize}
        \item discretisation error: 
        \begin{equation}
        \begin{aligned}
            O(h^2/\varepsilon_{n}) \leq \delta \quad &\Rightarrow \quad h = O(\sqrt{\delta \varepsilon_{n}}), \\
            O(e^{-\lambda_{\min}(\mathbf{A})t}) \leq \delta \quad &\Rightarrow \quad t=O(-\log(\delta)/\lambda_{\min}(\mathbf{A})), \\
            O(\Delta p) \leq \delta \quad &\Rightarrow \quad \Delta p = O(\delta).
        \end{aligned}
        \end{equation}

        \item quantum cost:
        \begin{equation}
        \begin{aligned}
            &\tau = s\|\mathbf{H}_{total}\|_{\max}t = O(3^{2d}dh^{-2}\log\delta^{-1}/\Delta p), \quad m_{\mathbf{H}_{total}} = O(d\log N + \log(1/\Delta p)), \\
            &\mathcal{Q}_{can} = O\left( \tau m_{\mathbf{H}_{total}} \cdot \mathrm{polylog} \right) = \tilde{O}\left(3^{2d} d^{\frac{11}{2}} \varepsilon_{n}^{-1} \delta^{-2}\right)=\tilde{O}\left(3^{2d} d^{\frac{11}{2}} \varepsilon_1^{-(n+2)}\right).
        \end{aligned}
        \end{equation}
    \end{itemize}

    \item \textbf{Homogenization model - Classical method}  
    \begin{itemize}
        \item discretisation error: 
        \begin{equation}
            O(h^2) \leq \delta \Rightarrow h = O(\sqrt{\delta})
        \end{equation}
        
        \item classical cost: 
        \begin{equation}
            \mathcal{C}_{hom} = O\left(s(\mathbf{\tilde{A}})\sum_{k=1}^{n+1}N^{kd}\sqrt{\kappa(\mathbf{\tilde{A}})}\log(1/\delta)\right) = \tilde{O}\left(3^{\frac{3(n+1)d}{2}}d^{\frac{1}{2}}\delta^{-\frac{(n+1)d+1}{2}}\right) = \tilde{O}\left(3^{\frac{3(n+1)d}{2}}d^{\frac{1}{2}}\varepsilon_1^{-\frac{(n+1)d+1}{2}}\right).
        \end{equation}
    \end{itemize}
    
    \item \textbf{Homogenization model - Quantum method}
    \begin{itemize}
        \item discretisation error: 
        \begin{equation}
        \begin{aligned}
            O(h^2) \leq \delta \quad &\Rightarrow \quad h = O(\sqrt{\delta}), \\
            O(e^{-\lambda_{\min}(\mathbf{\tilde{A}})t}) \leq \delta \quad &\Rightarrow \quad t=O(-\log(\delta)/\lambda_{\min}(\mathbf{\tilde{A}})), \\
            O(\Delta p) \leq \delta \quad &\Rightarrow \quad \Delta p = O(\delta).
        \end{aligned}
        \end{equation}
        
        \item quantum cost: 
        \begin{equation}
        \begin{aligned}
            &\tau = O(3^{2(n+1)d}dh^{-2}\log\delta^{-1}/\Delta p), \quad m_{\mathbf{H}_{total}} = O(nd\log N + \log(1/\Delta p)), \\
            &\mathcal{Q}_{hom} = O\left( \tau m_{\mathbf{H}_{total}} \cdot \mathrm{polylog} \right) = \tilde{O}\left(3^{2(n+1)d} n^{\frac{9}{2}}d^{\frac{11}{2}} \delta^{-2}\right)=\tilde{O}\left(3^{2(n+1)d} n^{\frac{9}{2}}d^{\frac{11}{2}} \varepsilon_1^{-2}\right).
        \end{aligned}
        \end{equation}
    \end{itemize}
\end{itemize}

\subsection{Multiscale parabolic model}
\begin{itemize}
    \item \textbf{Canonical model - Classical method} 
    \begin{itemize}
        \item discretisation error: 
        \begin{equation}
        \begin{aligned}
            O(h^2/\varepsilon_{1}) \leq \delta \quad &\Rightarrow \quad h = O(\sqrt{\delta \varepsilon_{1}}), \\
            O(\Delta t) \leq \delta \quad &\Rightarrow \quad \Delta t = O(\delta).
        \end{aligned}
        \end{equation}
        
        \item classical cost: \begin{equation}
            \mathcal{C}_{can} = O\left(s(\mathbf{A})N^dN_T\sqrt{\kappa(\mathbf{A})}\log(1/\delta)\right) = \tilde{O}\left(3^{\frac{3d}{2}}d^{\frac{1}{2}} (\delta\varepsilon_{1})^{-\frac{2d+1}{4}} T/\Delta t \right) = \tilde{O}\left(3^{\frac{3d}{2}}d^{\frac{1}{2}} \varepsilon_1^{-d-\frac{3}{2}}\right).
        \end{equation}
    \end{itemize}
    
    \item \textbf{Canonical model - Quantum method} 
    \begin{itemize}
        \item discretisation error: 
        \begin{equation}
        \begin{aligned}
            O(h^2/\varepsilon_{1}) \leq \delta \quad &\Rightarrow \quad h = O(\sqrt{\delta \varepsilon_{1}}), \\
            O(\Delta t) \leq \delta \quad &\Rightarrow \quad \Delta t = O(\delta), \\
            O(e^{-\lambda_{\min}(\mathbf{A})t}) \leq \delta \quad &\Rightarrow \quad t=O(-\log(\delta)/\lambda_{\min}(\mathbf{A})), \\
            O(\Delta p) \leq \delta \quad &\Rightarrow \quad \Delta p = O(\delta).
        \end{aligned}
        \end{equation}

        \item quantum cost:
        \begin{equation}
        \begin{aligned}
            &\tau = s\|\mathbf{H}_{total}\|_{\max}t = O(3^{2d}dh^{-1}\log\delta^{-1}/\Delta p), \quad m_{\mathbf{H}_{total}} = O(d\log N + \log N_T + \log(1/\Delta p)), \\
            &\mathcal{Q}_{can} = O\left( \tau m_{\mathbf{H}_{total}} \cdot \mathrm{polylog} \right) = \tilde{O}\left(3^{2d} d^{\frac{11}{2}} \varepsilon_{1}^{-\frac{1}{2}} \delta^{-\frac{3}{2}}\right)=\tilde{O}\left(3^{2d} d^{\frac{11}{2}} \varepsilon_1^{-2}\right).
        \end{aligned}
        \end{equation}
    \end{itemize}

    \item \textbf{Homogenization model - Classical method}  
    \begin{itemize}
        \item discretisation error: 
        \begin{equation}
        \begin{aligned}
            O(h^2) \leq \delta \quad &\Rightarrow \quad h = O(\sqrt{\delta}), \\
            O(\Delta t) \leq \delta \quad &\Rightarrow \quad \Delta t = O(\delta).
        \end{aligned}
        \end{equation}
        
        \item classical cost: 
        \begin{equation}
            \mathcal{C}_{hom} = O\left(s(\mathbf{\tilde{A}})(N^{2d}+N^d)N_T\sqrt{\kappa(\mathbf{\tilde{A}})}\log(1/\delta)\right) = \tilde{O}\left(3^{3d}d^{\frac{1}{2}}\delta^{-d-\frac{1}{2}}T/\Delta t\right) = \tilde{O}\left(3^{3d}d^{\frac{1}{2}}\varepsilon_1^{-d-\frac{3}{2}}\right).
        \end{equation}
    \end{itemize}
    
    \item \textbf{Homogenization model - Quantum method}
    \begin{itemize}
        \item discretisation error: 
        \begin{equation}
        \begin{aligned}
            O(h^2) \leq \delta \quad &\Rightarrow \quad h = O(\sqrt{\delta}), \\
            O(\Delta t) \leq \delta \quad &\Rightarrow \quad \Delta t = O(\delta), \\
            O(e^{-\lambda_{\min}(\mathbf{\tilde{A}})t}) \leq \delta \quad &\Rightarrow \quad t=O(-\log(\delta)/\lambda_{\min}(\mathbf{\tilde{A}})), \\
            O(\Delta p) \leq \delta \quad &\Rightarrow \quad \Delta p = O(\delta).
        \end{aligned}
        \end{equation}
        
        \item quantum cost: 
        \begin{equation}
        \begin{aligned}
            &\tau = O(3^{4d}dh^{-2}\log\delta^{-1}/\Delta p), \quad m_{\mathbf{H}_{total}} = O(nd\log N + \log N_T + \log(1/\Delta p)), \\
            &\mathcal{Q}_{hom} = O\left( \tau m_{\mathbf{H}_{total}} \cdot \mathrm{polylog} \right) = \tilde{O}\left(3^{4d}d^{\frac{11}{2}} \delta^{-2}\right)=\tilde{O}
            \left(3^{4d}d^{\frac{11}{2}} \varepsilon_1^{-2}\right).
        \end{aligned}
        \end{equation}
    \end{itemize}
\end{itemize}

\subsection{Multiscale wave equations}
\begin{itemize}
    \item \textbf{Canonical model - Classical method} 
    \begin{itemize}
        \item discretisation error: 
        \begin{equation}
        \begin{aligned}
            O(h^2/\varepsilon_{1}) \leq \delta \quad &\Rightarrow \quad h = O(\sqrt{\delta \varepsilon_{1}}), \\
            O(\Delta t) \leq \delta \quad &\Rightarrow \quad \Delta t = O(\delta).
        \end{aligned}
        \end{equation}
        
        \item classical cost: \begin{equation}
            \mathcal{C}_{can} = O\left(s(\mathbf{A})N^dN_T\sqrt{\kappa(\mathbf{A})}\log(1/\delta)\right) = \tilde{O}\left(3^{\frac{3d}{2}} d^{\frac{1}{2}} (\delta\varepsilon_{1})^{-\frac{d}{2}} T/\Delta t\right) = \tilde{O}\left(3^{\frac{3d}{2}} d^{\frac{1}{2}} \varepsilon_1^{-d-1}\right).
        \end{equation}
    \end{itemize}
    
    \item \textbf{Canonical model - Quantum method} 
    \begin{itemize}
        \item discretisation error: 
        \begin{equation}
        \begin{aligned}
            O(h^2/\varepsilon_{1}) \leq \delta \quad &\Rightarrow \quad h = O(\sqrt{\delta \varepsilon_{1}}), \\
            O(\Delta t) \leq \delta \quad &\Rightarrow \quad \Delta t = O(\delta), \\
            O(e^{-\lambda_{\min}(\mathbf{A})t}) \leq \delta \quad &\Rightarrow \quad t=O(-\log(\delta)/\lambda_{\min}(\mathbf{A})), \\
            O(\Delta p) \leq \delta \quad &\Rightarrow \quad \Delta p = O(\delta).
        \end{aligned}
        \end{equation}

        \item quantum cost:
        \begin{equation}
        \begin{aligned}
            &\tau = s\|\mathbf{H}_{total}\|_{\max}t = O(3^{2d}dh^{-1}\log\delta^{-1}/\Delta p), \quad m_{\mathbf{H}_{total}} = O(d\log N + \log N_T + \log(1/\Delta p)), \\
            &\mathcal{Q}_{can} = O\left( \tau m_{\mathbf{H}_{total}} \cdot \mathrm{polylog} \right) = \tilde{O}\left(3^{2d} d^{\frac{11}{2}} \varepsilon_1^{-\frac{1}{2}} \delta^{-\frac{3}{2}}\right)=\tilde{O}\left(3^{2d} d^{\frac{11}{2}} \varepsilon_1^{-2}\right).
        \end{aligned}
        \end{equation}
    \end{itemize}

    \item \textbf{Homogenization model - Classical method}  
    \begin{itemize}
        \item discretisation error: 
        \begin{equation}
        \begin{aligned}
            O(h^2) \leq \delta \quad &\Rightarrow \quad h = O(\sqrt{\delta}), \\
            O(\Delta t) \leq \delta \quad &\Rightarrow \quad \Delta t = O(\delta).
        \end{aligned}
        \end{equation}
        
        \item classical cost: 
        \begin{equation}
            \mathcal{C}_{hom} = O\left(s(\mathbf{\tilde{A}})(N^{2d}+N^d)N_T\sqrt{\kappa(\mathbf{\tilde{A}})}\log(1/\delta)\right) = \tilde{O}\left(3^{3d} d^{\frac{1}{2}}\delta^{-d-\frac{1}{2}}T/\Delta t\right) = \tilde{O}\left(3^{3d}d^{\frac{1}{2}}\varepsilon_1^{-d-\frac{3}{2}}\right).
        \end{equation}
    \end{itemize}
    
    \item \textbf{Homogenization model - Quantum method}
    \begin{itemize}
        \item discretisation error: 
        \begin{equation}
        \begin{aligned}
            O(h^2) \leq \delta \quad &\Rightarrow \quad h = O(\sqrt{\delta}), \\
            O(\Delta t) \leq \delta \quad &\Rightarrow \quad \Delta t = O(\delta), \\
            O(e^{-\lambda_{\min}(\mathbf{\tilde{A}})t}) \leq \delta \quad &\Rightarrow \quad t=O(-\log(\delta)/\lambda_{\min}(\mathbf{\tilde{A}})), \\
            O(\Delta p) \leq \delta \quad &\Rightarrow \quad \Delta p = O(\delta).
        \end{aligned}
        \end{equation}
        
        \item quantum cost: 
        \begin{equation}
        \begin{aligned}
            &\tau = O(3^{4d}dh^{-2}\log\delta^{-1}/\Delta p), \quad m_{\mathbf{H}_{total}} = O(nd\log N + \log N_T + \log(1/\Delta p)), \\
            &\mathcal{Q}_{hom} = O\left( \tau m_{\mathbf{H}_{total}} \cdot \mathrm{polylog} \right) = \tilde{O}\left(3^{4d}d^{\frac{11}{2}} \delta^{-2}\right)=\tilde{O}\left(3^{4d}d^{\frac{11}{2}} \varepsilon_1^{-2}\right).
        \end{aligned}
        \end{equation}
    \end{itemize}
\end{itemize}

\begin{myremark}
    We would like to clarify that our quantum cost analysis only includes the cost of quantum subroutines, specifically the cost of producing quantum states, and does not include the same data as in classical computation. Quantum measurement steps have  not been taken into account in our analysis. Therefore, the comparison we have made does not provide a complete picture of the actual quantum costs. The total quantum costs should also include the cost of encoding initial classical data into quantum states and the quantum measurement costs at output.
\end{myremark}

\section{Conclusion}
\label{sec:conclusion}

In this paper, we investigate the numerical methods for solving multiscale PDEs, including elliptic, parabolic and wave equations,  and compare the computational complexities of classical and quantum methods applied to both the canonical models and the homogenization models. 

We have conducted a thorough analysis of multiscale elliptic problems and demonstrated that quantum advantages can be achieved in both the canonical and homogenization model methods with respect to the scaling parameter $\varepsilon_1$. Specifically, the classical computations for both methods have more complexities than $\tilde{O}\left(\varepsilon_1^{-(n+1)d/2}\right)$, making the problem infeasible for large $n$ (the number of multiple scales) or $d$ (the spatial dimension). On the other hand, the quantum cost for the canonical model, with respect to $\varepsilon_1$, has a complexity of $\tilde{O}\left(\varepsilon_1^{-(n+2)}\right)$, which is {\it independent} of $d$ but still dependent on $n$. This makes it a suitable option for high-dimensional problems with only a few scales. However, for large $n$, the quantum algorithm for the homogenization model is highly recommended since its complexity depends only on $3^{2nd}\text{poly}(n)$.

A notable observation is that the constant $3^{(n+1)d}$ appears in both the classical and quantum costs of the homogenization model. This is attributed to the sparsity of $\mathbf{\tilde{A}}$, and has a considerable effect when the complexities are of the same order with respect to $\varepsilon_1$ as that of the canonical model. However, this factor can be disregarded when comparing the quantum costs of elliptic models, as it grows much slower than $\varepsilon_1^{-n}$.

To better understand the complexities of classical and quantum methods for multiscale parabolic and wave equations, we limit our analysis to two-scale problems with $n=1$. It is important to note that while estimates for multiscale elliptic equations are relatively well-developed, quantitative estimates for multiscale parabolic and wave problems with $n>1$ are not as advanced. Quantum advantages can be attained in all cases, as the order of complexity with respect to $\varepsilon_1$ in quantum algorithms remains {\it independent} of the dimensionality $d$. Remarkably, the homogenization does not help, and may even worsen the performance for parabolic and wave equations. If we consider the general $n$-scale case and refer to the results from elliptic equations, the classical cost is larger by a factor of $n/2$ with respect to $\varepsilon_1$. However, for parabolic and wave equations, the condition number of the canonical equation is typically one or two orders of magnitude smaller than that of the homogenized equation. This leads to a reduction in the classical canonical cost by $1/2$ or $1$ order. Our existing findings for $n=1$ are consistent with this observation. For $n\geq 2$, it is expected that the homogenization model will be less expensive than the canonical model, but this remains to be mathematically justified. 

The complexity results presented in this paper rely on the convergence of the finite element solution and the reiterated homogenization model.  It is possible to use higher-order finite element methods to obtain similar results. Nonetheless, establishing the higher-order corrector approximation of the multiscale ($n\geq 2$) homogenization solutions remains an open question.

\section*{Acknowledgments}
SJ thanks Christoph Schwab for some interesting references on elliptic homogenization. LZ thanks Wenjia Jing and Jinping Zhuge for stimulating discussions on quantitative homogenization results. SJ was partially supported by the NSFC grant No. 12031013, the Shanghai Municipal Science and Technology Major Project (2021SHZDZX0102). LZ was partially supported by the NSFC grant No. 12271360, the Shanghai Municipal Science and Technology Project 22JC1401600, and the Fundamental Research Funds for the Central Universities.

\appendix

\bibliographystyle{plain}
\bibliography{ref.bib}

\end{document}